\documentclass[article]{siamart190516}
\usepackage{amsmath,amssymb}
\usepackage{caption}
\usepackage{hyperref}
\usepackage[utf8]{inputenc}
\usepackage{mathtools}
\usepackage{color,graphicx}
\usepackage{thmtools, thm-restate}
\usepackage{enumitem}
\usepackage{stmaryrd}

\reversemarginpar

\usepackage{etoolbox}
\usepackage{mathabx}
\makeatletter

\patchcmd{\@addmarginpar}{\ifodd\c@page}{\ifodd\c@page\@tempcnta\m@ne}{}{}
\providecommand*{\cupdot}{%
  \mathbin{%
    \mathpalette\@cupdot{}%
  }%
}
\newcommand*{\@cupdot}[2]{%
  \ooalign{%
    $\m@th#1\cup$\cr
    \hidewidth$\m@th#1\cdot$\hidewidth
  }%
}
\providecommand*{\bigcupdot}{%
  \mathop{%
    \vphantom{\bigcup}%
    \mathpalette\@bigcupdot{}%
  }%
}
\newcommand*{\@bigcupdot}[2]{%
  \ooalign{%
    $\m@th#1\bigcup$\cr
    \sbox0{$#1\bigcup$}%
    \dimen@=\ht0 %
    \advance\dimen@ by -\dp0 %
    \sbox0{\scalebox{2}{$\m@th#1\cdot$}}%
    \advance\dimen@ by -\ht0 %
    \dimen@=.5\dimen@
    \hidewidth\raise\dimen@\box0\hidewidth
  }%
}
\makeatother
\reversemarginpar
\newsiamremark{expl}{Example}

\newcommand{\cO}{\mathcal{O}}

\newcommand{\cM}{\mathcal{M}}

\newcommand{\cR}{\mathcal{R}}

\newcommand{\cU}{\mathcal{U}}

\newcommand{\cX}{\mathcal{X}}

\newcommand{\bitm}{\begin{itemize}}
\newcommand{\eitm}{\end{itemize}}
\newcommand{\bitme}{\begin{enumerate}[label=(\roman*),leftmargin=0.25in]}
\newcommand{\eitme}{\end{enumerate}}
\newcommand{\beq}{\begin{equation}}
\newcommand{\eeq}{\end{equation}}
\newcommand{\btcb}{\begin{tcolorbox}}
\newcommand{\etcb}{\end{tcolorbox}}
\def\bals#1\eals{\begin{align*} #1 \end{align*}}
\def\bal#1\eal{\begin{align} #1 \end{align}}

\newcommand{\Tthen}{\quad \text{ then } \quad}

\newcommand{\where}{\quad \text{ where } }
\newcommand{\Ffor}{\quad \text{ for }}
\newcommand{\Aand}{\quad \text{ and } \quad }

\newcommand\Dom\Omega

\newcommand\RR{\mathbb{R}}

\newcommand\NN{\mathbb{N}}

\newcommand\UU{\mathbb{U}}
\newcommand\VV{\mathbb{V}}

\newcommand\Lap\Delta

\newcommand\abs[1]{\left\lvert #1 \right\rvert}

\newcommand\dx{\,\mathrm{d}x}
\newcommand\dy{\,\mathrm{d}y}

\newcommand\dtau{\,\mathrm{d}\tau}

\def\bpde#1\epde{\[\left\{\begin{aligned}#1\end{aligned}\right. \]}
\def\inbpde#1\inepde{\left\{\begin{aligned}#1\end{aligned}\right.}
\def\binpde#1\einpde{\left\{\begin{aligned}#1\end{aligned}\right.}

\newcommand\ddf[2]{\frac{\p {#1}}{\p {#2}}}

\newcommand\SemiNorm[2]{\lvert { #1 } \rvert_{#2}}
\newcommand\SemiNormlr[2]{\left\lvert { #1 } \right\rvert_{#2}}
\newcommand\Normlr[2]{\left\lVert { #1 } \right\rVert_{#2}}
\newcommand\Norm[2]{\lVert { #1 } \rVert_{#2}}

\def\cU{\mathcal{U}}

\def\cR{\mathcal{R}}
\def\half{\frac{1}{2}}

\def\cU{\mathcal{U}}

\def\p{\partial}

\def\b0{\mathbf{0}}

\def\eps{\varepsilon}

\def\bbmat{\begin{bmatrix}[r]}
\def\ebmat{\end{bmatrix}}

\newcommand{\barr}{\begin{array}}
\newcommand{\ea}{\end{array}}
\newcommand{\bea}{\begin{eqnarray}}
\newcommand{\eea}{\end{eqnarray}}
\newcommand{\bt}{\begin{table}}
\newcommand{\et}{\end{table}}

\DeclareMathOperator\Id{Id}
\DeclareMathOperator\Span{span}

\DeclareMathOperator\supp{supp}
\DeclareMathOperator\rg{range}

\DeclareMathOperator\diag{diag}

\newtheorem{example}[theorem]{Example}

\numberwithin{equation}{section}

\newcommand\BB{\mathbb{B}}
\newcommand\WW{\mathbb{W}}

\newcommand\tturl[1]{{\tt \scriptsize [\url{{#1}}]}}

\newcommand{\graycolor}[1]{{\color{gray}#1}}

\newcommand{\minialign}[1]{\left\{\begin{aligned} #1 \end{aligned}\right.}

\newcommand{\TheTitle}
{A Low Rank Neural Representation of \\Entropy Solutions}

\newcommand{\hDomx}{\widehat{\Dom}_x}
\newcommand{\Domx}{\Dom_x}
\newcommand{\Domt}{\Dom_t}

\newcommand{\Const}{C_{u_0, T, F}}

\newcommand{\hId}{\hat{I}}
\newcommand{\charX}{X}
\newcommand{\hG}{\widehat{G}}
\newcommand{\hcharX}{\widehat{X}}
\newcommand{\chkcharX}{\widecheck{X}}
\newcommand{\hcharXc}{\widehat{X}_0}
\newcommand{\shockset}{\Upsilon_t}
\newcommand{\shocksett}[1]{\Upsilon_{#1}}

\newcommand{\skleft}{{\vdash}}
\newcommand{\skright}{{\dashv}}

\newcommand{\otimesd}{\otimes_d}
\newcommand{\dotimes}{{}_d\otimes}

\usepackage{titlesec}

\begin{document}

\ifpdf
\DeclareGraphicsExtensions{.pdf, .jpg, .tif}
\else
\DeclareGraphicsExtensions{.eps, .jpg}
\fi

\title{\TheTitle}

\author{%
  Donsub Rim\thanks{{\scriptsize Department of Mathematics, %
  Washington University in St. Louis, St. Louis, MO 63130 ({\tt
  rim@wustl.edu})}}%
  \and 
  Gerrit Welper%
  \thanks{{\scriptsize Department of Mathematics, University of Central Florida,
  Orlando, FL 32816 ({\tt gerrit.welper@ucf.edu})}}
}

\maketitle

\begin{abstract}
  We construct a new representation of entropy solutions to nonlinear scalar
  conservation laws with a smooth convex flux function in a single spatial
  dimension.  The representation is a generalization of the method of
  characteristics and posseses a compositional form.  While it is a nonlinear
  representation, the embedded dynamics of the solution in the time variable is
  linear. This representation is then discretized as a manifold of implicit
  neural representations where the feedforward neural network architecture has a
  low rank structure. Finally, we show that the low rank neural representation
  with a fixed number of layers and a small number of coefficients can
  approximate any entropy solution regardless of the complexity of the shock
  topology, while retaining the linearity of the embedded dynamics. 
\end{abstract}

  
\begin{AMS}
\texttt{68T07}, \texttt{41A46}, \texttt{41A25}, \texttt{65N15}, \texttt{35L65}
\end{AMS}

\maketitle

\section{Introduction}

Current state-of-the-art numerical algorithms simulating realistic fluid
dynamics, even when run on the best hardware available today, are not real
time. One proposed strategy in accelerating these solvers is to seek
approximations of solutions to partial differential equations (PDEs) with
small degrees of freedom. A class of such approximations called
\emph{reduced models} have been successfully applied for solutions to
certain parametrized PDEs \cite{Cohen_DeVore_2015,Hesthaven2016}.

The reduced models are commonly constructed as linear low-dimensional
approximations written as a superposition of a few basis functions, and we
will refer to such models as \emph{classical reduced models}. When such an
approximation to the solution to a time-dependent problem is possible, the
computational complexity required to evolve the solution forward in time
often scales with the number of the basis functions
\cite{Hesthaven_Pagliantini_Rozza_2022}. For example, if standard
numerical algorithms such as finite difference methods require polynomial
in $N$ complexity to compute one forward time step according to the PDE,
where $N$ is the number of finite difference grid points, a reduced model
using $\log N$ basis functions would instead require the reduced
polylogarithmic complexity in $N$ to do so, allowing the simulations to
run faster than real time. This difference in the reduced complexity is a
key desired feature in reduced models.

The limitations of classical reduced models are readily seen when one
attempts to produce one for the simplest of hyperbolic problems, even the
advection equation \cite{rowley00,Constantine2012}. The cause is the slow
decay rate of the Kolmogorov width \cite{pinkus12} of the solution
manifold; its lower bounds imply that there cannot be a small set of
reduced basis functions whose superpositions can approximate the solution
uniformly well \cite{Welper2017,Ohlberger16,Greif19}. More generally,
non-classical reduced models are also not completely immune from such
lower bounds; slow decay can be established even for nonlinear benchmarks
such as the stable width and entropy for high-dimensional parametric
transport equations \cite{Cohen2022, rim2023performance}. Also relevant in
this context are studies of stability of classical reduced models
\cite{Haasdonk_Ohlberger_2008,Nguyen2009,Dahmen_Plesken_Welper_2014,amsallem}.

Various attempts were made to obtain the reduced complexity for
convection-dominated problems by incorporating the structure of the
transport. To provide a possibly incomplete list of proposed approaches
that are the most relevant to this work: template-fitting \cite{rowley00},
method of freezing \cite{Beyn04,Ohlberger13}, shock reconstruction
\cite{Constantine2012}, approximated Lax-Pairs \cite{Gerbeau14}, advection
modes \cite{iollo14}, transported snapshot interpolation (TSI)
\cite{Welper2017}, shifted proper orthogonal decomposition (sPOD)
\cite{schulze18}, calibrated manifolds \cite{CagniartMadayStamm2019},
Lagrangian basis method \cite{Mojgani17}, transport reversal
\cite{rim17reversal}, registration methods \cite{taddei2020}, Wasserstein
barycenters \cite{ehrlacher19} or quadratic manifolds
\cite{GeelenWrightWillcox2023}. Some of these proposed methods can be
viewed as dynamical low-rank approximations \cite{koch07,Sapsis2009}. 

However, relatively few methods we encountered \cite{NairBalajewicz2019,
  taddei2021, rim2023mats} have demonstrated (1) it can be deployed in a
numerical scheme to solve the PDE using the approximation (e.g. in
Galerkin systems), and (2) has an online complexity matching that of
classical reduced models. These works exploit a nonlinear model that
composes two superpositions of a few basis functions. We will call the
nonlinear class of reduced models that can be represented in compositional
forms satisfying the two criteria \emph{complexity-separated compositional
  reduced models}, or simply \emph{compositional reduced models}.

The compositional reduced models, while achieving the two properties
required of classical reduced models, have not yet demonstrated an ability
to handle shock propagation in their full generality; they were tested in
restricted settings without shock interaction. This is a serious
limitation, since for many PDEs arising from fluid dynamics, shock
formation and collision are defining features. For example, a typical
entropy solution to the Burgers' equation is shown in
Fig.~\ref{fig:burgers_waterfall}. Shocks can form at various times, travel at
varying speeds, and merge; new techniques are required to handle these
challenges. For example, a computational approach called Front Transport
Reduction (FTR) employs level sets to represent the changes in the shock
topology \cite{Krah2023}.

\begin{figure}
  \centering
  \includegraphics[width=0.8\textwidth]{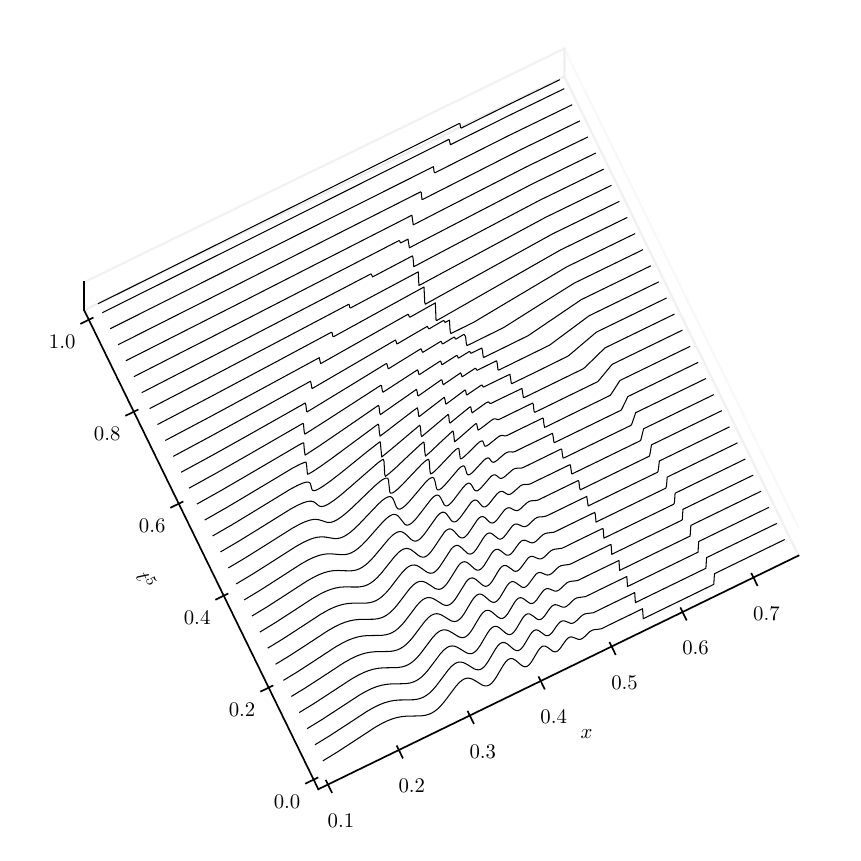}
  \caption{An example of solution to the Burgers' equation. Time variable has
    been rescaled to better show details.}
  \label{fig:burgers_waterfall}
\end{figure}

The dearth of progress invites the question: What is the essential
obstacle that prevents the handling of shocks for the compositional
models? Part of the answer lies in the fact that while it is clear that
the characteristic curves can be approximated well by a few basis
functions in convective problems when, for example, the velocity is
smooth, it is not at all clear what low-dimensional structure can be
exploited in the presence of various shock configurations.

In another line of approach, there have been attempts to leverage deep
learning models to overcome the difficulty
\cite{HanJentzenWeinan2018,KhooLuYing2020,MohanGaitonde2018,HesthavenUbbiali2018,WangHesthavenRay2019,gonzalez2018deep,lee19,RegazzoniDedeQuarteroni2019,Mojgani17,Fresca2021,Mojgani2023,Barnett2023},
but whether one can devise deep learning models that can be evolved
according to the PDE with the reduced complexity remains an open problem.
The currently existing models do not achieve the complexity found in
successful classical reduced models \cite{Kim2022}. Given the
compositional models studied in the past, and given how the deep learning
models are defined with layers of compositions, it is natural to ask if
there is a variant of the standard feedforward architecture that can
connect them.

There are previous theoretical results relating transport equations,
reduced models, and neural networks, in a broad sense. An approximation of
parametric linear transport equations was derived in the setting of smooth
convective fields \cite{Laakmann2021} by approximating the initial
condition and the characteristic curves \cite{Yarotsky2017,Petersen2018}.
Neural networks were shown to be capable of emulating the approximations
of reduced basis methods \cite{Geist2021,Kutyniok2022}. A more recent
result \cite{dahmen2023compositional} showed various estimates using the
notion of \emph{compositional sparsity}; e.g. when the high-dimensional
parametric convection field has an affine decomposition, it was shown that
there is a neural network approximation whose number of weight parameters
depend linearly on the affine dimension. 

In this work, we show one can exploit an inherent nonlinear
low-dimensional structure in the entropy solution by repeated
compositions, and present a theoretical construction in the form of a
family of feedforward deep neural networks. 

Neural networks that take spatial-temporal inputs are referred to as
implicit neural representations (INRs) in deep learning or computer vision
literature (e.g. \cite{Park2019, Sitzmann2020}); we follow suit and refer
to our neural networks as INRs. Our construction can be viewed as a
continuously parametrized family of INRs whose weight parameters are
constrained to belong to a fixed low-dimensional linear space, and has the
following desirable properties: (1) The construction is a direct
generalization of the previous classical and compositional models, (2) the
construction can approximate entropy solutions to nonlinear scalar
conservation laws with arbitrarily many shocks and their interactions
without an increase in the reduced dimensionality, and (3) the embedded
temporal dynamics in the weights and biases of the INR is linear.

The first property is a key contribution of this work, showing that the
previous compositional model that achieves online-offline decomposition
\cite{rim2023mats} can be expressed as 2-layer INR of this form, making
the deeper analogue also a promising candidate model for achieving the
reduced complexity. 

The third property has important practical implications: Despite the
involved spatial approximation, the temporal dependence of the weight
parameters is linear. This suggests that this model has the potential to
approximate complicated solutions while maintaining regular dependence on
time. Such a feature would enable one to computationally integrate the
representation in time, regardless of the lack of regularity in the
spatial variable due to shocks. We briefly mention here a close relation
between regular-in-time approximation introduced here and classical
numerical methods exploiting similar properties, e.g. Large Time Step
methods \cite{largetimestep1} or Averaged Multivalued Solutions
\cite{Brenier1984}: These methods exploit the fact that classical
characteristic curves are regular in time, that is, they are straight lines.
These methods propagate the characteristics for longer durations than permitted
by the Courant-Friedrichs-Lewy condition and then modify them to ensure proper
shock propagation. We closely follow this principle of time-regular
characteristic evolution with minimal modification. In contrast, standard
finite-volume methods alter the characteristic curves at each time step and do
not preserve their straightness or temporal regularity \cite{LeVeque2002}.

Another benefit of the temporal linearity is the favorable space-time
approximation rates for entropy solutions in single spatial dimension,
which are necessarily of bounded variation. Classical approximation
methods require $\mathcal{O}(\epsilon^{-2})$ degrees of freedom to achieve
$L_1$ error $\epsilon$ for such functions. In contrast, our results rely
on the regularity of input data and the flux function, but not on
specialized regularity classes such as cartoon-like functions
\cite{Petersen2018,Bolcskei2019,devore_hanin_petrova_2021,
  MarcatiOpschoorPetersenSchwab2022,BernerGrohsKutyniokPetersen2022}. That
is, the input data consisting of initial value and the nonlinear flux
function are single dimensional and can be approximated with only
$\mathcal{O}(\epsilon^{-1})$ degrees of freedom. While the maximum width
of our networks matches the degrees of freedom of classical method, we
only require $\mathcal{O}(\epsilon^{-1})$ trainable weights due to the
specialized architecture, and it is also possible to achieve
$\mathcal{O}(1)$ evaluation time per point on a space-time grid.

We note an earlier work using reduced deep networks (RDNs) to approximate
solutions to nonlinear wave problems \cite{rim2020depth}, but the
approximation here is an improvement in two important aspects: (1) In the
earlier work, a bisection algorithm was implemented within the
architecture using the feedforward layers which increased the number of
requisite layers depending on the desired accuracy, and we avoid this here
and the accuracy depends primarily on the width, and (2) The number of
reduced dimensions do not depend on the number of shocks. As only
speculated therein, this work illustrates a more concrete evidence of a
new kind of depth separation, one that involves dimensionality reduction
for deeper neural networks.

More recently, a continuously parametrized family of INRs was successfully
used in the Physics-Informed Neural Networks (PINNs) framework
\cite{Raissi2019,Cai2021,Cuomo2022}, and exhibited improved convergence
behaviors for solving PDEs that were challenging for the original PINNs
\cite{Krishnapriyan2021,cho2023hypernetworkbased}. This general approach
was referred to as Low-Rank PINNs (LR-PINNs).

Our construction here is a similarly parametrized family of INRs with
low-rank architecture we call the Low-Rank Neural Representation (LRNR).
We use the term LRNR to distinguish the technical definition of the
architecture alone, as separate from the machine learning context, and to
contrast with RDNs \cite{rim2020depth} which took on a different approach
to approximating the weak solution. Readers familiar with the LR-PINNs may
view LRNR as a version of the architecture used in Low-Rank PINNs
\cite{cho2023hypernetworkbased}, keeping in mind that the time variable is
treated as a parameter here rather than an input.

The trained LRNRs can serve as a reduced model approximation of the family
of solutions to the parametrized PDEs: A rapid approximation is achieved
due to the small number of degrees of freedom, much like the rapid computations
performed during the online phase for classical reduced models. Moreover,
it is not difficult to see that a reduction in online complexity can be
achieved by utilizing the low rank structure in the weights and biases%
\footnote{During the review of this work, a concrete result demonstrating this
observation has appeared in \cite{Cho2025}.}
much like the classical reduced models \cite{Hesthaven2016}. This is a
property LRNR inherits by being a direct generalization of the
compositional reduced model utilizing transported subspaces
\cite{rim2023mats}.

Due to the compositional structure, a connection can be made to the
optimal transport problem that is associated with conservation laws
\cite{bolley05,ehrlacher19,rim2023mats}. Here, we depart from that
connection and achieve an efficient representation of characteristic
curves arising from nonlinear conservation laws without relying on optimal
transportation.

\subsubsection*{Organization of the Paper}
The subject of this paper spans three different mathematical topics:
The PDE theory of scalar conservation laws, reduced models, and neural networks.
This paper is organized to cover all three aspects while aiming to keep each
section as self-contained as possible.

\begin{itemize}[leftmargin=10pt, topsep=4pt, itemsep=2pt]
  \item \emph{Section~\ref{sec:entropy_theory}.} This section introduces new representation of the entropy
  solutions.
  The PDE theoretical presentation is self-contained and covers the entropy
  solution in its full generality; that is, solutions with arbitrary shock and
  rarefaction waves.

  \item \emph{Section~\ref{sec:lrnr}.} A new neural network architecture called LRNR is introduced, inspired
  by the form of the new entropy solution representation from the previous
  section.

  \item \emph{Section~\ref{sec:LRNR_classical}.} We discuss the close relation between a subclass of LRNRs
  and the
  compositional reduced model called transported subspaces. We 
  develop an important technique called the inverse-bias trick, and we prove
  that classical solutions to scalar conservation laws can be approximated
  efficiently by a LRNR with two layers and rank three.
  This section also serves as an introduction of new
  ideas with relatively light technical details.

  \item \emph{Section~\ref{sec:LRNR_approx_entropy}.} We develop a LRNR
  approximation to the entropy solutions by
  building on the new representation of the entropy solution developed in
  Section 2 and the inverse-bias trick from Sec.~\ref{sec:LRNR_classical}.  We show that
  general entropy solutions can be approximated by a LRNR with five layers and
  rank two. This construction establishes a general 
  low dimensional representation of
  the entropy solution, enabling an efficient representation of arbitrary shock
  wave interactions.
\end{itemize}

This paper introduces a substantial amount of new formulations and definitions.
To aid the readers, the definitions and constructions are listed in a
glossary in Appendix~\ref{sec:glossary}. We will refer the reader to the
glossary throughout the paper. Moreover, for brevity of presentation,
the proofs of lemmas and theorems are postponed to
Appendix~\ref{sec:proofs}.


\section{A compositional form of entropy solutions}
\label{sec:entropy_theory}

In this section, we introduce the compositional forms of the entropy solution to
nonlinear scalar conservation laws. We will first introduce the PDE and the
generalized characteristics as well as the entropy solution. Then we introduce 
\emph{rarefied characteristics}, which enables a representation of the entropy
solution with a single composition without a separate description of the 
rarefaction fans. Finally, we derive yet another alternative representation of the
entropy solution using \emph{relief characteristics} which reveals a certain low
rank structure in the entropy solution.

\subsection{Entropy solutions} \label{sec:entropy_solution}

In this section, we will describe the PDE that will be the focus of the
paper.

For domains $\Domx, \Domt \subset \RR$ let us denote by $C(\Domx)$ the
space of continuous functions, by $C(\Domt; L^1(\Domx))$ space of
continuous functions from $\Domt$ to $L^1(\Domx)$, by $P_0(\Domx)$
piecewise constant functions, and by $P_1(\Domx)$ piecewise linear
functions.

Throughout, our spatial domain will be the unit interval $\Domx := (0,
  1)$, and the temporal domain the interval $\Domt := (0, T)$ for some $T
  \in \RR_+$. The space-time domain is denoted by $\Dom := \Domx \times
  \Domt$. We consider the scalar conservation law for $u: \Dom \to \RR$,
\begin{equation} \label{eq:ivp}
  \left\{
  \begin{aligned}
    \partial_t u + \partial_x (F \circ u) & = 0,
    \quad (x, t) \in \Dom,                                      \\
    u(\cdot\,, 0)                         & = u_0,              \\ 
    u(0\,, \cdot)                         & = u(1\,,\cdot) = 0, \\ 
  \end{aligned}
  \right.
\end{equation}
where the initial value $u_0$ lies in
\begin{equation} \label{eq:cU}
  \cU
  := 
  \{ v \in BV(\Domx) \mid v \text{ has compact support }\},
\end{equation}
and we assume a smooth flux function $F \in C^\infty(\RR)$ that is strictly
convex. We choose $T$ to be small enough so that the zero boundary
conditions are satisfied for all $t \in \Domt$.

For each $u_0 \in \cU$, there exists $\Domt$ for which there is a solution
$u \in C(\Domt; L^1(\Domx))$ such that $u(\cdot, t) \in \cU$ for all $t
  \in \Domt$. See standard texts, for example \cite{Ser99,
  LeVeque2002,Dafermos2010}, for proofs and references as well as numerical
approximations.

It is possible to construct a convergent sequence to the solution $u$ by
considering simpler initial conditions, e.g. piecewise-constant initial
conditions with finitely many jumps. For each $u_0 \in \cU$ there is a
sequence of initial $u_0^{(i)}$ that is a member of
\begin{equation} \label{eq:u0space}
  \overline{\cU}
  :=
  \{
  v \in \cU \cap P_0(\Domx) 
  \mid
  v \text{ has finitely many jumps }
  \}
\end{equation}
whose corresponding solution $u^{(i)}(\cdot, t)$ converge to $u(\cdot, t)$ in
$L^1(\Domx)$ uniformly for all $t \in \Domt$. This fact is used for
constructing approximations to the solution; see for example \cite{Daf72,Dafermos2010}. 
As a result, we can focus on initial conditions $u_0 \in \overline{\cU}$. In
this case, it is known that the solution $u$ only has discontinuities along
finitely many rectifiable curves in $\Dom$ \cite{Ser99}.

Unlike the familiar setting in which one \emph{seeks} solutions to the
initial value problem knowing only the initial condition $u_0$, our focus
here is to \emph{reformulate} the entropy solution in a simplified form.
Accordingly, we will freely make use of the knowledge of the solution $u$.

Assuming that the solution $u \in C(\Domt ; L^1(\Domx))$ is known, we will
first describe the characteristic curves. Let us define the function $G:
  \Domx \times \Domt \to \RR$, 
\begin{equation} \label{eq:char_ode}
  G(\charX, t)
  :=
  \left\{
  \begin{aligned}
    F'(u(\charX, t)) 
     & 
    \quad \text{ if } \llbracket u(\charX, t) \rrbracket = 0, \\
    \frac{\llbracket F(u(\charX, t)) \rrbracket}
    {\llbracket u(\charX, t) \rrbracket}
     & 
    \quad \text{ if } \llbracket u(X, t) \rrbracket \ne 0,    \\
  \end{aligned}
  \right. 
\end{equation}
where $\llbracket v(x, t) \rrbracket := v(x_+, t) - v(x_-, t)$ denotes the
different between right and left limits of $v(\cdot, t)$ at $x$.
The function $G$ yields the velocity of the classical characteristic curves if the
point $(X, t) \in \Dom$ lies in a smooth part of $u$, or yields the shock speed
via the Rankine-Hugoniot jump conditions if it lies on a shock curve where $u$
is discontinuous. 

\emph{Characteristic curves} of the initial value problem \eqref{eq:ivp} are
given by the solutions $X: \Domx \times \Domt \to \RR$ to the family of ordinary
differential equations
\begin{equation} \label{eq:charcurves}
  \begin{aligned}
    \frac{\p \charX}{\p t}
    =
    G(\charX, t),
    \quad
    \charX(x, 0) = x,
    \quad
    (x, t) \in \Domx \times \Domt.
  \end{aligned}
\end{equation}

This formulation for the characteristic curves is also called
\emph{generalized characteristics} \cite{Dafermos2010}. It is
straightforward to see that the problem \eqref{eq:charcurves} has a unique
solution. This is due to the following: The RHS $G$ is locally continuous,
so the unique solution exists locally where it is continuous \cite{Hal80};
the solution is uniformly stable in neighborhoods of the points of
discontinuity due to the entropy conditions, and at points of
discontinuity the shock curve is also uniquely determined by the
Rankine-Hugoniot jump conditions \cite{Ser99}. One also observes that the
unique solution $X$ is a continuous function of the initial value $x$, as
$G$ is of bounded variation.

Let us define the modified version of $G$ by letting $G: \Domx \times
  \Domx \times \Domt \to \RR$ be
\begin{equation} \label{eq:modcharcurves}
  G(\charX, x, t)
  := 
  \left\{
  \begin{aligned}
    F'(u_0(x)) 
     & 
    \quad \text{ if } \llbracket u(\charX, t) \rrbracket = 0, \\
    \frac{\llbracket F(u(\charX, t)) \rrbracket}
    {\llbracket u(\charX, t) \rrbracket}
     & \quad 
    \text{ if } \llbracket u(\charX, t) \rrbracket \ne 0.     \\
  \end{aligned}
  \right. 
\end{equation}
We refer to both functions \eqref{eq:charcurves} and \eqref{eq:modcharcurves}
as $G$ and distinguish them by the number of the independent
variables (i.e. number of input arguments).

Since our problem has no source terms \eqref{eq:ivp}, the solution is constant along
the characteristic curves. So we may rewrite \eqref{eq:charcurves} as
follows,
\begin{equation} \label{eq:charcurves_hmg}
  \frac{\p \charX}{\p t}
  =
  G(\charX, x, t),
  \quad
  \charX (x, 0) = x.
\end{equation}

It is well-known that the characteristic curves may not fill all of
$\Dom$, so a similarity solution is sought to obtain the solution in all
of the space-time domain $\Dom$ \cite{LeVeque2002}. The set of points in
the space-time domain $\Dom$ that cannot be connected to a point in $\Domx
  \times \{0\}$ by a characteristic curve is given by
\begin{equation}
  \begin{aligned}
    \cR := \left\{
    (x, t) \in \Dom 
    \mid 
    x \in \Domx \smallsetminus X(\Domx, t) 
    \right\}.
  \end{aligned}
\end{equation}
We will reformulate the similarity solutions as follows.  Note that there exist
two functions $\xi: \cR \to \RR$ and $\overline{\xi} : \cR \to \Domx$ with the
relation 
\begin{equation} \label{eq:xi}
  \xi(x,t) = \frac{x - \overline{\xi}(x,t)}{t},
  \quad
  (x, t) \in \Dom,
\end{equation}
where $\overline{\xi}$ is a function that yields the origin of the rarefaction
fan in $\Domx$. This function is well-defined, since points in the rarefaction
fan always have a unique origin \cite{Ser99}.

We define a right inverse for nondecreasing real-valued function $g$
defined on a real interval, 
\begin{equation} \label{eq:left-inv}
  g^+(y) := \inf \{x \mid g(x) = y\}.
\end{equation}
Throughout, the right inverse will always be applied to the spatial variable:
That is, for functions of both variables $x \in \Domx$ and $t \in \Domt$, it
will be applied solely to the spatial variable $x$.

The entropy solution $u$ is then given in all of $\Domx \times \Domt$ by
the formula
\begin{equation} \label{eq:entropy}
  u(x, t)
  =
  \left\{
  \begin{aligned}
    u_0 ( \charX^+ (x, t)) 
    \quad & \text{ if } (x, t) \in \Dom \smallsetminus \cR,
    \\
    (F')^{-1} \left( \xi(x,t) \right)
    \quad & \text{ if } (x, t) \in \cR.
  \end{aligned}
  \right.
\end{equation}

\subsection{Rarefied characteristics}

In this section, we modify the differential equation describing the
characteristic curves \eqref{eq:charcurves_hmg}, so that the form of the
entropy solution \eqref{eq:entropy} can be simplified. The modified form of
the solution combines the two cases into just one. We shall achieve this
by extending the initial condition $u_0$ and the characteristic curves $X$
to an extended spatial domain that is larger than the original spatial
domain $\Domx$.

Observe that the set $\overline{\xi}(\cR)$ is 
a finite union of single points, 
since there are only
finitely many jumps in the piecewise constant initial condition $u_0 \in
  \overline{\cU}$ from which the rarefaction fans originate. Let $n_\cR :=
  |\overline{\xi}(\cR)|$ then we order and enumerate these by
\begin{equation}
  \overline{\xi}(\cR)
  =
  \left\{ \xi_1 < \xi_2 < ... < \xi_{n_\cR} \right\}.
\end{equation}
Correspondingly, we define a partition of $\cR = \bigcup_{k=1}^{n_\xi} \cR_{k}$
by the following sets,
\begin{equation}
  \cR_{k}
  :=
  \left\{
  (x, t) \in \cR \mid \overline{\xi}(x,t) = \xi_{k}
  \right\}.
\end{equation}
They are subsets of cones in $\Dom$ associated with a
rarefaction fan emanating from a point $(\xi_{k}, 0) \in \Dom$.

Let us define a collection of $n_\cR$ positive numbers $\eps_{k} =
  u_0(\xi_{k+}) - u_0(\xi_{k-})$ so that $\sum_{k=1}^{n_\cR} \eps_{k} =
  \abs{u_0}_{PV(\Domx)}$, the positive variation of $u_0$ \cite{Holden2002}
given by
\begin{equation}
  \abs{u_0}_{PV(\Domx)}
  =
  \sup_{x_0 < x_1 < \cdots < x_N}
  \sum_{i}
  \max\{u_0(x_i) - u_0(x_{i-1}), 0\},
\end{equation}
where the supreumum is taken over any set of points 
$x_0 < x_1 < \cdots < x_N$ lying in $\Domx$ and for any $N \in \NN$.  Then for each $k \in \NN$, define the largest and
smallest characteristic velocities in the rarefaction fans by
\begin{equation}
  \begin{aligned}
    \nu_{k}^+ 
    := 
    \sup_{(x, t) \in \cR_{k}} \frac{x - \xi_{k}}{t},
    \qquad
    \nu_{k}^- 
    := 
    \inf_{(x, t) \in \cR_{k}} \frac{x - \xi_{k}}{t}.
  \end{aligned}
\end{equation}
Note that $\nu_k^{\pm}$ necessarily take on finite values \cite{Ser99}.

Then we define an extended spatial domain $\hDomx$ as an interval given by
\begin{equation}\label{eq:ext_dom}
  \hDomx := \left(0, \, 1 + \sum_{k=1}^{n_\cR} \eps_k \right)
  =
  \left( 0, 1 + \abs{u_0}_{PV(\Domx)} \right).
\end{equation}
We embed $\Domx$ into $\hDomx$ via the map $\iota: \Domx \to \hDomx$
\begin{equation}
  \iota(x) := x + \sum_{\{k : \xi_{k} < x\}} \eps_{k},
\end{equation}
and assign intervals $\Gamma_k$ of length $\eps_{k}$ starting at $\iota(\xi_k)$,
\begin{equation} \label{eq:Gammax}
  \Gamma_{k} := \iota(\xi_{k}) +  (0,\, \eps_{k}],
  \qquad
  \Gamma_x := \bigcup_{k=1}^{n_\cR} \Gamma_{k}.
\end{equation}
That is, the extended domain $\hDomx$ is a longer interval where $\Gamma_{k}$
have been inserted into points in $\{\xi_k\} \subset \Domx$.  That is, $\hDomx =
  \iota (\Domx) \cup \Gamma_x$ and $\hDomx$ is larger than $\Domx$ by some length 
\begin{equation}
  \abs{\Gamma_x}
  =
  \sum_{k=1}^{n_\cR} |\Gamma_{k}| 
  = 
  \sum_{k=1}^{n_\cR} \eps_{k}
  =
  \abs{u_0}_{PV(\Domx)},
\end{equation}
where $\abs{\cdot}$ denotes the measure of a set.
Next we extend the initial condition $u_0$ to $\hDomx$ by letting $\hat{u}_0 :
  \hDomx \to \RR$ be given as
\begin{equation} \label{eq:uhat}
  \hat{u}_0 \circ \iota := u_0,
  \qquad
  \hat{u}_0 \Big\rvert_{\Gamma_{k}}
  :=
  (F')^{-1} \left(
  \nu^-_{k}
  +
  \frac{\nu^+_k - \nu^-_k}{\eps_k} (x - \iota(\xi_k))
  \right).
\end{equation}
Then $\hat{u}_0$ has no positive jumps (entropy violating shocks \cite{Ser99}) 
and has the same total variation as $u_0$, 
\begin{equation}
  |\hat{u}_0|_{TV(\Domx)} = |u_0|_{TV(\hDomx)}.
\end{equation}
Intuitively speaking, one obtains $\hat{u}_0$ from $u_0$ by replacing positive
jumps with continuous slopes, prescribed the similarity solution.  
See Fig.~\ref{fig:u0extend} for an illustration comparing $u_0$ and $\hat{u}_0$.

\begin{figure}
  \centering
  \includegraphics[width=1.0\textwidth]{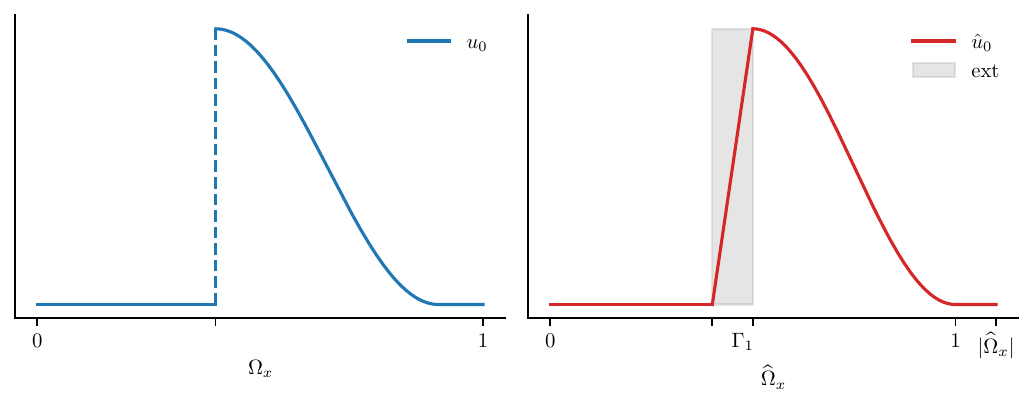}
  \caption{An example of an extended initial condition $\hat{u}_0$ of $u_0$
  in the case of Burgers' flux.
  The newly assigned values of $\hat{u}_0$ in the interval $\Gamma_1$ are
  given by \eqref{eq:uhat}. The values are linear due to the choice of
  the Burgers' flux.}
  \label{fig:u0extend}
\end{figure}

We define \emph{rarefied characteristic curves} $\hcharX: \hDomx \times
  \Domt \to \Domx$ as the solution to the differential equation
\begin{equation} \label{eq:rarefied_charcurves}
  \begin{aligned}
    \frac{\p \hcharX}{\p t}
    =
    \hG (\hcharX, x, t),
    \quad
    \hcharX(x, 0) = \hId(x),
  \end{aligned}
\end{equation}
where $\hId: \hDomx \to \Domx$ is a continuous piecewise linear function satisfying $\hId^+ = \iota$ given by 
\begin{equation}\label{eq:hId}
  \hId (0) = 0, 
  \quad 
  \hId' = \mathbf{1}_{\hDomx \setminus \Gamma_x},
\end{equation}
and $\hG: \Domx \times \hDomx \times \Domt \to \RR$ is an extension of $G$ \eqref{eq:modcharcurves},
\begin{equation} \label{eq:ext_modcharcurves}
  \hG(\hcharX, x, t)
  := 
  \left\{
  \begin{aligned}
    F'(\hat{u}_0(x)) 
     & 
    \quad \text{ if } \llbracket u(\hcharX, t) \rrbracket = 0, \\
    \frac{\llbracket F(u(\hcharX, t)) \rrbracket}
    {\llbracket u(\hcharX, t) \rrbracket}
     & \quad 
    \text{ if } \llbracket u(\hcharX, t) \rrbracket \ne 0.     \\
  \end{aligned}
  \right. 
\end{equation}
The nonlinear RHS term \eqref{eq:ext_modcharcurves} of the ODE
\eqref{eq:rarefied_charcurves} can be written as a function of bounded variation
that depends only on $\hcharX$ and $t$. 
This implies that $\hcharX(\cdot, t)$ is non-decreasing for all $t \in
  \Domt$, regardless of whether $u_0$ has jumps (see, e.g. Ch. 5 of
\cite{Hal80}); positive jumps of $u_0$ are replaced by continuous pieces
in $\hDomx$, negative jumps are replaced by piecewise constants.

See Fig.~\ref{fig:rarefaction} for an illustration showing $\hId$ and its role
in relating $u_0$ with $\hat{u}_0$ \eqref{eq:rarefied_charcurves}. A
diagram depicting the evolution of $\hcharX$ is shown in
Fig.~\ref{fig:hcharX}.

\begin{figure}
  \centering
  \includegraphics[width=0.9\textwidth]{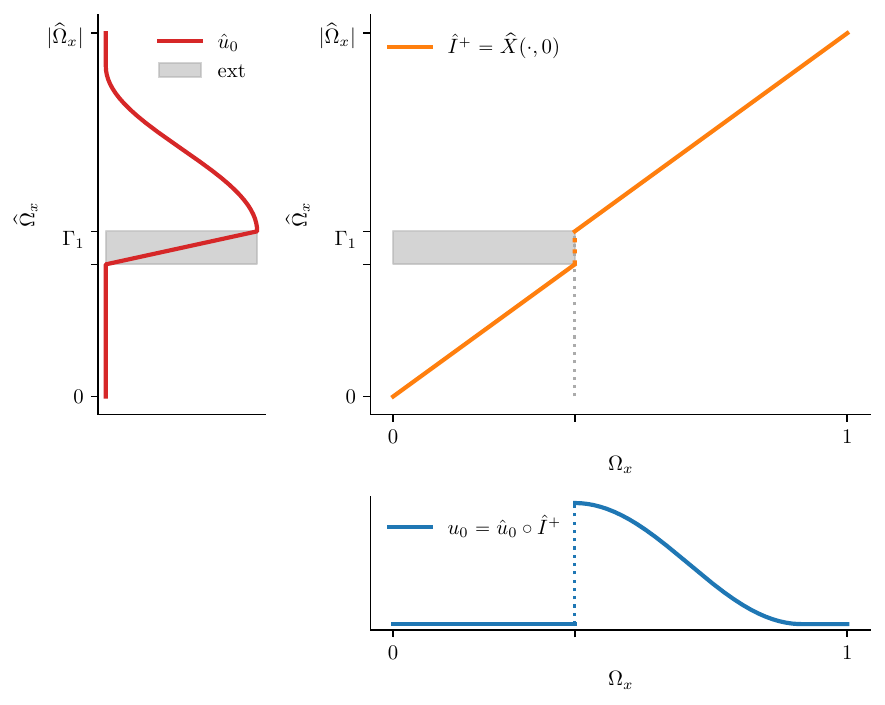}
  \caption{An example of the extension $u_0$ to $\hat{u}_0$ as in
    \eqref{eq:uhat} for the Burgers flux $F(u) = u^2 / 2$.
    The function $\hId^+$ has a jump discontinuity of magnitude $|\Gamma_1|$
      as depicted in the upper right plot. Due to this jump, the values of
      $\hat{u}_0$ inside interval $\Gamma_1$ highlighted in the upper left plot are
      effectively omitted in the composed function $\hat{u}_0 \circ \hId^+$. 
      This effect is
      shown in the lower right plot.}
  \label{fig:rarefaction}
\end{figure}

\begin{lemma}\label{lem:rarefied}
  The entropy solution $u$ to \eqref{eq:ivp} is given by
  \begin{equation} \label{eq:u0Yp}
    u(x, t) = \hat{u}_0 \circ \hcharX^+(x,t).
  \end{equation}
  Recall that $(\cdot)^+$ denotes the left-inverse \eqref{eq:left-inv}.
\end{lemma}

\begin{proof} See Appendix~\ref{proof:lem:rarefied}. \end{proof}

We make a couple of remarks here. Note that new shocks cannot form at
$\hcharX(x_0, t)$ for $x_0 \in \Gamma_x$ since the solution to
\eqref{eq:rarefied_charcurves} is strictly increasing with respect to $x$
at $x_0$. So the solution given by \eqref{eq:u0Yp} at the points
$(\hcharX(x_0, t), t)$, for times before $\hcharX(x_0, \cdot)$ enters a
shock, agrees with the expression for the rarefaction fan. These
observations imply that $\hcharX$ is continuous.

Let us define the \emph{rarefied characteristic curves with multivalued
  inverse},
\begin{equation} \label{eq:hcharX0}
  \hcharX_0(x, t) := \hat{I}(x) + t (F'\circ \hat{u}_0) (x)
\end{equation}
which is in $\Span\{ \hId, F' \circ \hat{u}_0 \}$. This map represents the
classical characteristic curves and the rarefaction waves but represents the
overturned multivalued solution, instead of using shocks \cite{LeVeque2002}.

\begin{figure}
  \centering
  \includegraphics[width=0.9\textwidth]{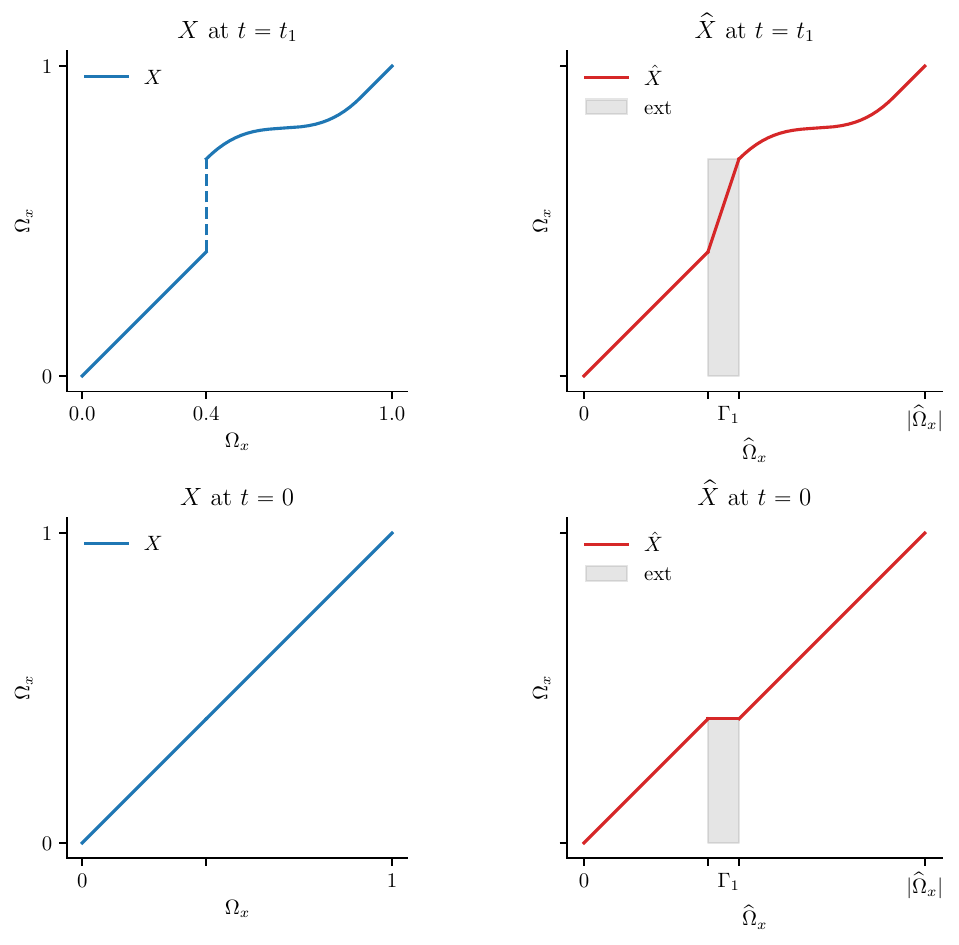}
  \caption{An illustration comparing the characteristic curves $X$ (left) with
    the evolution of rarefied characteristics $\hcharX$ at some later time $t_1 > 0$
    (right).}
  \label{fig:hcharX}
\end{figure}

\subsection{Shock time function} \label{sec:shock_time}

Here we will define the shock time function, which we will be a crucial
component in our constructions in the later sections.

Let us define the parametrized set we call the \emph{shock set} $\shockset$ for $t \in
  \Domt$ to contain all $x \in \hDomx$ such that the preimage of the
singleton $\{x\}$ given by $\hcharX^{-1}( \hcharX(\{x\} ,\, t),\, t)$ is
non-unique. That is, 
\begin{equation}
  \shockset 
  :=
  \left\{ x \in \hDomx
  \mid 
  \exists y \in \Domx 
  \text{ such that } y \ne x, \hcharX(x,t) = \hcharX(y, t)
  \right\}.
\end{equation}
Then $(\shockset)_{t \in \Domt}$ is non-decreasing with respect to $t \in
  \Domt$, i.e. for $0 \le t_1 \le t_2 \le T$ we have $\shocksett{t_1} \subset
  \shocksett{t_2}$ due to the entropy conditions which ensure the characteristics
merge into the shock.  Since $\hcharX$ is a non-decreasing 
continuous function of $x$ for all
$t$ as noted above, $\shockset$ is closed for all $t \in \Domt$.
In plain terms, the set $\shockset$ are points in the extended spatial
domain $x_0 \in \hDomx$ such that, the rarefied characteristic curve leaving
from the point $\hcharX(x_0, \cdot\,)$ has merged into a shock at a time before
$t$.

We define a function $\lambda: \shocksett{T} \to \Domt$ we call the
\emph{shock time function}, given in terms of its level sets.
Specifically, let $\lambda$ be given by
\begin{equation} \label{eq:Yt}
  \left\{x \in \hDomx \mid \lambda(x) \le t \right\} = \shockset.
\end{equation}
Note that $\shockset$ is closed 
for all $t \in \Domt$; 
$\widehat{X}(\cdot, t)$ is non-decreasing and continuous, 
implying that $\shockset$ is 
a union of disjoint closed intervals. It follows that
$\lambda$ is continuous: a closed set $D$ lying in $\Domt$ is compact,
and moreover $\shockset$ is increasing in $t$, hence we have $\lambda^{-1} (D) =
  \shocksett{\max D}$, which is closed.

In intuitive terms, the value $\lambda(x)$ is the time the characteristic
curve starting from $x \in \shocksett{T} \subset \hDomx$ merges into a
shock.

It is left to the reader to show that $\hcharX$ is related to $\hcharX_0$ by
the relation,
\begin{equation} \label{eq:phi0}
  \hcharX(x, t)
  =
  \sup_{\substack{y < x \\ y \in \hDomx \setminus \shockset }}
  \hcharX_0(y, t),
\end{equation}
yielding an alternative formulation of $\hcharX$ that reveals it is
simply $\hcharX_0$ with certain portions replaced by a constant regions.

\subsection{Compositional form}
\label{sec:entropy_compositional}

In the previous section, we have simplified the expression of the entropy
solution from \eqref{eq:entropy} to \eqref{eq:phi0}, by
incorporating the evolution of the rarefaction waves into the
characteristic curves in \eqref{eq:rarefied_charcurves}. Next, we simplify
how the shock is represented in $\hcharX$. 

In the case of classical solutions, 
the characteristic curves \eqref{eq:charcurves}
are represented by a linear superposition,
\begin{equation} \label{eq:classical_superpose}
  X(x, t) = x + t (F' \circ u_0) (x)
  \quad
  \in
  \quad
  \Span \{ \Id, F' \circ u_0 \}.
\end{equation}
We wish to rewrite $\widehat{X}(x, t)$ in a form with similar linear
superpositions, modulo compositions and left-inverses.
Note that one can show the family of functions $\{ \hcharX(\cdot, t) \mid t \in
  \Domt \}$ can have a slowly decaying Kolmogorov width, except in
special cases with simple initial data $u_0$. 
Consider the stationary shock example shown in Fig.~\ref{fig:stationary}. The
shock is represented by the constant region $\shockset$ in $\widehat{X}$, which
is expanding over time. The jump in the derivative in $\widehat{X}$ travels to
the left and the right, which cannot be efficiently represented by the
superposition of a few basis functions.

\begin{figure}
  \centering
  \includegraphics[width=0.9\textwidth]{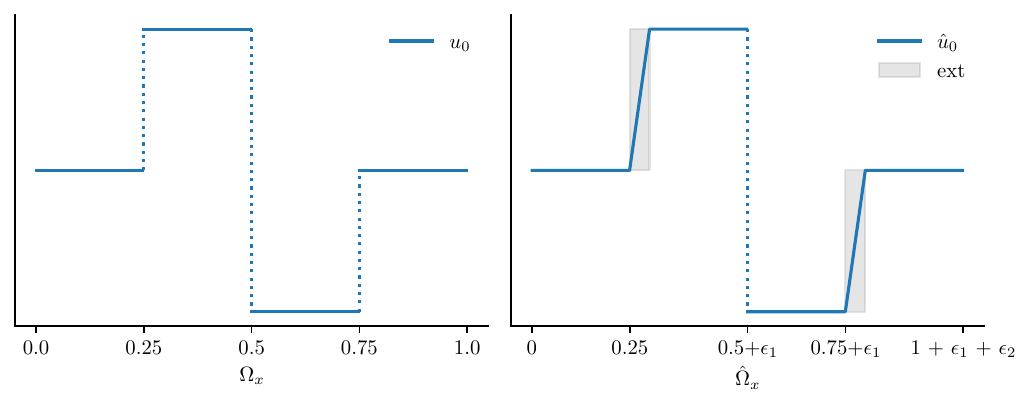}
  \\
  \includegraphics[width=0.7\textwidth]{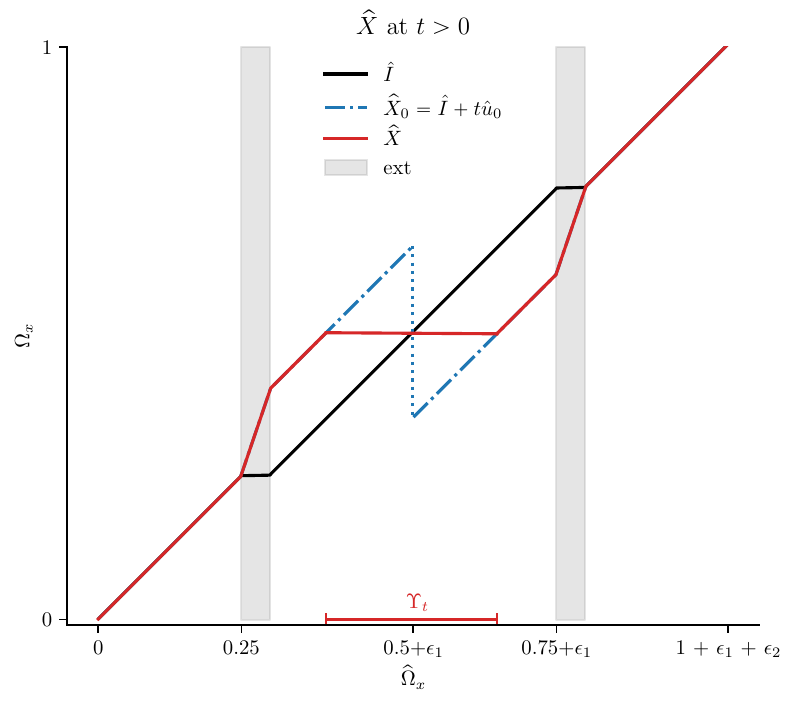}
  \caption{A stationary shock example for the Burgers' equation.
      Due to the jump discontinuity in the
      initial condition $\hat{u}_0$ located at $x=0.5$, 
      the shock forms immediately at
      initial time. This shock is represented as the constant part of the function
      $\hcharX$ in $\shockset$, and the left-inverse of the rarefied characteristic
      $\hcharX^+$ has a jump discontinuity there. As time evolves, the
      characteristics merge into the shock, so $\shockset$ where $\hcharX$ is
      constant correspondingly expands over time. Consequently, the kinks in
      $\hcharX$ at the left and right endpoints of $\shockset$ travels to the left
      and to the right, respectively. (For a plot of the evolution of the rarefied characteristics, see Fig.~\ref{fig:stationary_movie} below.)}
  \label{fig:stationary}
\end{figure}

Our key idea is to cut out the portion of the domain where $\hcharX_0(
  \cdot, t)$ represents the shock, by adding the characteristic function of
$\shockset$ to it. Let us denote the threshold function by
 $\rho := \mathbf{1}_{\RR_+}.$
We call the modified curves the \emph{relief characteristic curves}, given
by
\begin{equation} \label{eq:relief}
  \chkcharX
  =
  \hId 
  +
  t (F' \circ \hat{u}_0)
  +
  C_x \rho \circ (t - \lambda),
\end{equation}
for some sufficiently large constant $C_x \ge \abs{\Domx}$. Then, the entropy
solution is written as
\begin{equation} \label{eq:ucts}
  u(x, t)
  =
  \hat{u}_0
  \circ
  \chkcharX^+ (x, t).
\end{equation}
Once the left-inverse is applied to $\chkcharX$, the resulting map $\chkcharX^+
  : \Domx \to \hDomx$ is constant at points in $\shockset$, because the term
$\rho(t - \lambda(x))$ introduces the constant region and removes the portion of
$\hcharX$ which lies in $\shockset$ by adding a sufficiently large constant to
$\hcharX_0$ whenever it is evaluted at a point in $\shockset$. See
Fig.~\ref{fig:chop} for an illustration.

Written in a generic form, the solution is 
\begin{equation}\label{eq:ucts_generic}
  u(\cdot\,, t)
  =
  \psi_1 \circ 
  \left( \phi_1 + t \phi_2 + \phi_3 \circ (t \varphi_1 - \varphi_2) \right)^+
\end{equation}
where the individual functions are defined as
\begin{equation}
  \psi_1 = \hat{u}_0, 
  \quad
  \phi_1 = \hId, 
  \quad 
  \phi_2 = F' \circ \hat{u}_0, 
  \quad 
  \phi_3 = C_x \rho,
  \quad 
  \varphi_1 = 1, 
  \quad 
  \varphi_2 = \lambda.
\end{equation}
We will call this the \emph{compositional form of the entropy solution}.

In this representation, the time-dependence is in the coefficients of
linear combinations within each composition, and the dependence of the
coefficients with respect to $t$ is linear, in a similar manner as the
classical case. To elaborate, the first linear combination in
  \eqref{eq:ucts_generic} is $t \varphi_1 -
    \varphi_2$ with two coefficient values $t$ and $-1$. Only the first
  coefficient depends on $t$ and the dependence is linear.

In the compositional form, the left-inverse $(\cdot)^+$ persists, leaving
it an implicit form rather than an explicit one. When we consider the
discretization of this form of the entropy solution, we will make use of a
trick we call the \emph{inverse-bias trick} that will allow us to avoid
any direct evaluation of the left-inverse (Sec.~\ref{sec:ibtrick}).

We refer the reader to Appendix~\ref{sec:glossary} for a glossary
of various entropy solution formulations introduced in this section.

\begin{figure}
  \centering
  \includegraphics[width=0.5\textwidth]{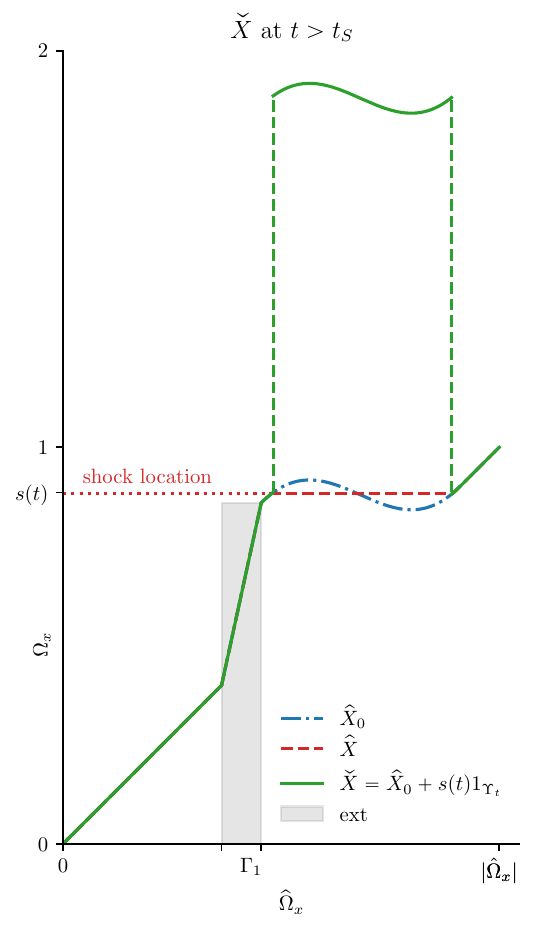}
  \caption{A diagram depicting the relief characteristics $\chkcharX$ after
    shock formation time $t_S$  is shown, along with the rarefied
      characteristics with multivalued inverse $\hcharX_0$. The relief
      characteristics use a different representation of the shock compared to
      the rarefied characteristics $\hcharX$, one that uses the addition of a box
      function $1_{\shockset}$. One observes that left-inverse $\chkcharX^+$ has
      a jump at $s(t)$, and this jump is identical to the corresponding one for
      $\hcharX^+$.}
  \label{fig:chop}
\end{figure}

\section{Low Rank Neural Representation} \label{sec:lrnr}

We introduce a family of feedforward neural networks that satisfy a low
rank condition. The condition is inspired by the fact that the entropy
solution, represented using the relief characteristic curves
\eqref{eq:ucts_generic}, is made up of compositions of low rank
representations. To make this point clear, we draw up a diagram of each
step of the evaluation of \eqref{eq:ucts_generic} along with the associated
linear spaces, as follows. 
\begin{equation} \label{eq:low_rank_upto_compositions}
  \begin{aligned}
               & ~x
    \\                    & \downarrow \qquad \searrow
    \\
               & \,x \qquad  t \varphi_1 - \varphi_2 
               &                                                       &  & 
    \Span\{\Id, \varphi_1, \varphi_2\}
    \\                    & \downarrow \qquad \quad \,\,\,\, \downarrow
    \\ \phi_1  + t \phi_2 & + \phi_3 \circ (t \varphi_1 - \varphi_2)
               &                                                       &  & 
       \Span\{\phi_1, \phi_2, \phi_3 \}
    \\                    & \downarrow
    \\ \psi_1 \circ
    (\phi_1  + & t \phi_2 + \phi_3 \circ (t \varphi_1 - \varphi_2))^+
               &                                                       &  & \Span\{\psi_1 \}
  \end{aligned}
\end{equation}
For example, the operations in the second step can be written as
\begin{equation} \label{eq:step_two}
  \begin{bmatrix}
    \mid      & \mid      & \mid 
    \\
    \phi_1(z) & \phi_2(z) & \phi_3(z) 
    \\
    \mid      & \mid      & \mid 
  \end{bmatrix}
  \begin{bmatrix}
    1 \\ t \\ 1
  \end{bmatrix}
  \quad
  \in
  \Span\{\phi_1, \phi_2, \phi_3 \}
\end{equation}
where $z = (x, t \varphi_1(x) - \varphi_2(x))^T$ is the output from the previous
step. This expression is superposition of a few 
vectors, if one views the function values $\phi_i(z)$ as column vectors of
infinite length. In addition, note that the individual the coefficients 
$[1, t, 1]^T$ has either a linear dependence on $t$ or is constant.

This structure of repeated compositions 
closely resembles the repeated compositions in a neural network. The 
component that requires a close scrutiny is the left-inverse $(\cdot)^+$
and its treatment will be a key discussion in the coming sections
(Sec.~\ref{sec:ibtrick} and Sec.~\ref{sec:LRNR_entropy_overview}). For the moment,
let us focus on establishing an analogy with feedforward neural networks.
We begin by introducing notations and definitions of a neural network. Our
notations are similar to those in common use (see, e.g.
\cite{devore_hanin_petrova_2021, BernerGrohsKutyniokPetersen2022}, for
similar definitions). Throughout, the superscript $\ell$ will index the
layer number.

\begin{definition}[Feedforward neural network architecture]
  Given the dimensions $(M_0, M_1, ...\ ,M_L) \in \NN^{L}$ let
  \begin{equation}
    \WW^{\ell}
    := 
    \{ W : \RR^{M_{\ell - 1}} \to \RR^{M_\ell} \mid W \text{ linear }\},
    \qquad
    \BB^{\ell} := \RR^{M_\ell},
  \end{equation}
  and define the Cartesian products
  \begin{equation}
    \WW := \bigtimes_{\ell = 1}^L \WW^{\ell},
    \qquad
    \BB := \bigtimes_{\ell = 1}^L \BB^{\ell}.
  \end{equation}
  We call the pair $(\WW, \BB)$ a \emph{feedforward neural network architecture}
  or simply \emph{architecture}. We refer to $M = \max_\ell M_\ell$ as the
  \emph{width} and $L$ as the \emph{depth} of the architecture.
\end{definition}

\begin{definition}[Feedforward Neural Networks] Given an architecture $(\WW,
    \BB)$, a \emph{feedforward neural network (NN)} with weights $W = (W^1, ...\ , W^L)
    \in \WW$ and biases $B = (B^1, ...\ , B^L) \in \BB$ is defined as the function
  $h: \RR \to \RR$ with
  \begin{equation}\label{eq:NN}
    h(x; W, B)
    :=
    A^L \circ \sigma \odot A^{L-1} \circ ... \circ \sigma \odot A^1(x)
  \end{equation}
  in which $\odot$ is the entrywise composition, $\sigma$ is the rectified linear unit (ReLU) $\sigma(x) = \max\{0, x\}$, and the affine maps $A_\ell$ are given by
  \begin{equation} \label{eq:NN_affine}
    A^\ell (z) = W^\ell z + B^\ell,
    \qquad 
    W^\ell \in \WW^\ell, 
    \quad
    B^\ell \in \BB^\ell.
  \end{equation}
  We will use the shorthand $h(W, B) := h(\ \cdot\ ; W, B)$.
  \label{def:nn}
\end{definition}

In what follows, we make the observation that the
superpositions in between the compositions in
\eqref{eq:low_rank_upto_compositions} is analogous to the affine
mappings that appear in the neural network \eqref{eq:NN_affine}. With this
viewpoint, we show it is natural to consider a neural network with the affine
mappings with low rank weights and biases, in correspondence to the low
rank representation in \eqref{eq:low_rank_upto_compositions}. 
Let us consider the composition of two functions which are individually in the
span of $\{\phi_i\}_{i=1}^r$ and $\{\varphi_i\}_{i=1}^r$, 
with the coefficients
$\{\gamma_i^3\}_{i=1}^r$ and $\{\gamma_i^2\}_{i=1}^r$, resp., 
\begin{equation} \label{eq:two_compositions}
 \left( \sum_{i=1}^r \gamma^3_i \phi_i \right)
 \circ
 \left( \sum_{i=1}^r \gamma^2_i \varphi_i \right) (x),
\end{equation}
which is a general form of the first composition in
\eqref{eq:low_rank_upto_compositions}. For the simple scenario $r = 1$
and both $\phi_1, \varphi_1: \RR \to \RR$ are scalar valued 1d
functions
representable by 2-layer NNs, their
composition has the structure
shown in black and grey below,
\begin{equation}
  \begin{aligned}
  &(\gamma_1^3 \phi_1) \circ (\gamma_1^2 \graycolor{\varphi_1}) (x)
  \\
  & =
  \gamma_1^3
  \begin{bmatrix} \cdots\cdots \end{bmatrix}
  \sigma \odot\, 
  \Bigg( 
    \underbrace{
    \gamma_1^2
  \begin{bmatrix}
    \vdots \\ \vdots 
  \end{bmatrix}
  \graycolor{\begin{bmatrix} \cdots\cdots \end{bmatrix}}}_{W^2}
  \graycolor{
  \sigma
  \odot\,
  \Bigg(
  \begin{bmatrix}
    \vdots \\ \vdots 
  \end{bmatrix}
  }
  \begin{bmatrix} x \end{bmatrix}
  \graycolor{
  +
  \begin{bmatrix}
    \vdots \\ \vdots 
  \end{bmatrix}
  \Bigg)
  + }
  \begin{bmatrix}
    \vdots \\ \vdots 
  \end{bmatrix}
  \graycolor{ [ \cdot ]}
  \Bigg)
  + [ \cdot ].
  \end{aligned}
\end{equation}
In particular, the weight matrix of the second layer is rank one:
\begin{equation}
  W^2 = \gamma_1^2 U_1^2,
  \qquad
  U^2_1 \text{ is rank-1}.
\end{equation}
Likewise, multiple scalar functions $\{\phi_i\}_{i=1}^r$ 
and $\{\varphi_i\}_{i=1}^r$
lead one to feedforward neural networks
\eqref{eq:NN} whose weights and biases \eqref{eq:NN_affine} are formed
as linear combinations of rank-1 matrices in each layer.
Assuming $r \ll M$, we have arrived at a low-rank form of the weight
matrix.

This derivation illustrates that, in
looking for a low-dimensional description of entropy solutions, 
it is natural to turn our attention to neural
networks whose weights $W$ and biases in $B$ are represented as linear
combinations of a few rank-1 linear transformations. For example, 
we require that
\begin{equation} \label{eq:lincomb_weights}
  \begin{aligned}
     & W^\ell = \sum_{i=1}^r \gamma_i^\ell U_i^\ell,
    \quad
    B^\ell = \sum_{i=1}^r \theta_i^\ell V_i^\ell,
    \qquad \quad
    \left\{
      \begin{aligned}
        \gamma_i^\ell, &\theta_i^\ell \in \RR,
                          \quad
      \\                   
      U_i^\ell &\in \WW^\ell,
                          V_i^\ell \in \BB^\ell  \text{ are rank-1,}
      \\
      \ell &= 1, ...\, , L.
      \end{aligned}
    \right.
  \end{aligned}
\end{equation}
The coefficients $\{\gamma^\ell_i, \theta^\ell_i\}$ correspond to the
coefficients $\{1, t\}$ used in the linear combinations at each compositional
step in \eqref{eq:low_rank_upto_compositions}.  One notes that the latter
coefficients depend on the parameter $t$ (the time variable), and the role of
this parameter is different from either the input $x$ or the fixed weights and
biases $\{U_i^\ell, V_i^\ell \}$ prescribing the linear spaces the weights
and biases belong to. It is therefore natural to view the coefficients
$\{\gamma_i^\ell, \theta_i^\ell \}$ as distinct external parameters; a
specific coefficient set $\{\gamma_i^\ell, \theta_i^\ell\}$ correspond
to a specific neural network belonging to the family of neural networks whose
weights and biases satisfy \eqref{eq:lincomb_weights}. The dependence on $t$ in
\eqref{eq:low_rank_upto_compositions} would correspond to the dependence of the
coefficients $\{\gamma_i^\ell, \theta_i^\ell \}$ on $t$.
On the other hand, the fixed weights and biases correspond to the 
functions in the low rank representation that do not depend on $t$, e.g.
the functions $\{\varphi_i, \phi_j, \psi_k\}$ in \eqref{eq:step_two}.

Motivated by these observations, we define a family of neural networks in
which (1) all of its members have weights and biases that lie in a fixed
low-dimensional linear space, and (2) the particular coefficients that
determine a specific neural network within that family is treated as a
special external parameter.

Let us set a notation for linear combinations. Given a linear space $\VV$,
let $\Phi_r = (\phi_i)_{i=1}^r \subset \VV$ and $A \subset \RR^r$, then
let
\begin{equation}
  A \cdot \Phi_r
  := 
  \left\{ 
  \sum_{i=1}^r \alpha_i \phi_i \mid \alpha \in A 
  \right\}.
\end{equation}
We denote by $\otimes$ the Kronecker product,
and introduce the Hadamard-Kronecker products 
$\otimes_d :
  \RR^{m_1} \times \RR^{m_2} \to \RR^{m_1 m_2 \times m_2}$
and 
${}_d\otimes : \RR^{m_1} \times \RR^{m_2} \to \RR^{m_1 \times m_1 m_2}
$ between two vectors as follows,
\begin{equation}
  u^L \otimesd u^R
  :=
  u^L \otimes \diag(u^R),
  \qquad
  u^L \dotimes u^R
  :=
  \diag(u^L) \otimes u^R.
\end{equation}

We now define the family of neural networks that is central to this paper.

\begin{definition}[Low Rank Neural Representation (LRNR)] \label{def:LRNR}
  Let us be given
  \begin{enumerate}[label=(\roman*)]

    \item an architecture $(\WW, \BB)$ of depth $L$ and width $M$,

    \item members of $\WW_\ell$ and $\BB_\ell$ for each $\ell = 1, ...\,,L$,
          \begin{equation}
            \UU^\ell_r := (U^\ell_i)_{i=1}^r \subset \WW^\ell,
            \qquad
            \VV^\ell_r := (V^\ell_i)_{i=1}^r \subset \BB^\ell,
          \end{equation}
          in which $U^\ell_i$ and $V^\ell_i$ are the products of two vectors
          \begin{equation}
            \left\{
            \begin{aligned}
              U^\ell_i
               & =
              u^{\ell,L}_i \bullet u^{\ell, R}_i,
              \\
              V^\ell_i
               & =
              v^{\ell,L}_i \bullet' v^{\ell, R}_i,
            \end{aligned}
            \right.
            \where \quad
            \bullet, \bullet' \in \{\otimes , {}_d\otimes, \otimes_d \},
          \end{equation}

    \item coefficient sets $C^\ell, D^\ell \subset \RR^r$ collected in
          \begin{equation}
            C
            :=
            \bigtimes_{\ell = 1}^L C^\ell,
            \qquad
            D
            :=
            \bigtimes_{\ell = 1}^L D^\ell.
          \end{equation}

  \end{enumerate}
  Let us denote
  \begin{equation} \label{eq:wrbr}
    \WW_r
    :=
    \bigtimes_{\ell = 1}^L (C^\ell \cdot \UU^\ell_r) \subset \WW,
    \qquad
    \BB_r
    :=
    \bigtimes_{\ell = 1}^L (D^\ell \cdot \VV^\ell_r) \subset \BB.
  \end{equation}
  Then a LRNR is defined as the parametrized family of feedforward neural
  networks
  \begin{equation} \label{eq:lrnr}
    H_r 
    := 
    \left\{
    h (W, B) \mid W \in \WW_r,\ B \in \BB_r 
    \right\}.
  \end{equation}
  We denote by $h (\gamma, \theta)$ the NN $h \in H_r$ that is
  associated with $(\gamma, \theta) \in C \times D$.
  For the LRNR $H_r$ we shall refer to $r$ as its \emph{rank}. 
  Denote by $K$ the \emph{degrees of freedom} (d.o.f), that is,
  the number of trainable or stored parameters, $M$ its \emph{width}, and $L$
  its \emph{depth}. One necessarily has $K \lesssim r L M$.

\end{definition}

The d.o.f refer to the set of parameters appearing in the weights or
biases that is unconstrained. In deep learning models, weight parameters
are often constrained in various ways, e.g. to be equal to each other, and
these constrained parameters are commonly referred to as \emph{shared}
weights \cite{Goodfellow2016}. As an example, a one dimensional
convolution layer with a single channel and a kernel of width three has
d.o.f $K=3$ but the width of that layer $M$ may be different.

We remark that in the LR-PINNs \cite{cho2023hypernetworkbased} the
coefficient $\theta_{i, \ell}$ for the biases were treated as fixed, and
the coefficients $\gamma_{i, \ell}$ were functions of the parameters of
the PDE.

In summary, a LRNR is a family of feedforward networks that have weights
and bias at each individual layer that belong to a fixed linear subspace
of dimension at most $r$. The aim of the following sections is to
establish the connection between LRNR (Def.~\ref{def:LRNR}) and the entropy
solution \eqref{eq:low_rank_upto_compositions}. Throughout this work, our
focus is on scalar problems in a single spatial dimension, so the input
and output dimensions are one, that is, $M_0 = M_L = 1$.

We refer the reader to Appendix~\ref{sec:glossary} for a
glossary of various notations and specifications regarding LRNRs introduced in
this section. 

\section{LRNR approximation of classical solutions} \label{sec:LRNR_classical}

In this section, we will discuss how a LRNR can be used to approximate
classical solutions to \eqref{eq:ivp}. In this case, the solution is given
by the method of characteristics \eqref{eq:classical_superpose}, 
and is known to have an efficient
approximation using transported subspaces \cite{rim2023mats}.
For the classical case, it is mandatory to suspend the assumption that the initial
condition is piecewise constant with finitely many jumps 
(that is, $u_0 \in \overline{\cU}$). Then, 
for a smooth initial condition $u_0 \in \cU$ \eqref{eq:cU} the classical
solution is 
\begin{equation} \label{eq:ivp_classical}
u(x,t) = u_0 \circ X^{-1}(x, t),
\end{equation}
with classical characteristics $X$ that solve the ODE (simpler version of \eqref{eq:char_ode}), 
\begin{equation} \label{eq:classical_charcurves}
  \begin{aligned}
    \frac{\p \charX}{\p t}
    =
    F'(u(\charX, t)),
    \quad
    \charX(x, 0) = x,
    \quad
    (x, t) \in \Domx \times \Domt.
  \end{aligned}
\end{equation}
Due to the absence of rarefaction and shock waves in the classical case, the
solutions simplifies in two major ways: 
(1) The inverse $(\cdot)^{-1}$ is used in the
solution \eqref{eq:ivp_classical} rather than the left-inverse $(\cdot)^{+}$
used in (\ref{eq:u0Yp}, \ref{eq:ucts}),
and (2) the initial condition $u_0$ and the characteristic curves $X$ are 
directly used instead of their extensions (i.e. $\hat{u}_0$, $\hcharX$ or
$\chkcharX$).

We focus on the fact that the classical solution 
\eqref{eq:ivp_classical} resembles the compositions of superpositions
\eqref{eq:two_compositions} we have encountered when deriving LRNRs
in the previous section. We may write the classical solution in a similar form,
\begin{equation}
  \begin{aligned}
    u(x, t)
    &=
    \left( \sum_{i=1}^r \gamma^2_i(t) \phi_i \right)
    \circ
    \left( \sum_{i=1}^r \gamma^1_i(t) \varphi_i \right)^{-1} (x)
    \\
    &\qquad \text{ with }
    r = 3
    \quad \text{ and } \quad
    \left\{
    \begin{aligned}
    \gamma^1_1 &= 1,
    &
    \gamma^1_2 &= t,
    &
    \gamma^1_3 &= 0,
    \\
    \varphi_1 &= \Id
    &
    \varphi_2 &= F' \circ u_0,
    &
    \\
    \gamma^2_1 &= 1,
    &
    \gamma^2_2 &= 0,
    &
    \gamma^2_3 &= 0,
    \\
    \phi_1 &= u_0.
    \end{aligned}
  \right.
  \end{aligned}
\end{equation}
An important difference with \eqref{eq:two_compositions} 
is the presence of the inverse $(\cdot)^{-1}$.
This must be carefully handled
to approximate this compositional form of the classical solution using 
LRNRs. This issue 
arose in the compositional reduced models and other neural network
approximations \cite{taddei2020,Laakmann2021,rim2023mats}. 
Our discussion here centers specifically on transported 
subspaces \cite{rim2023mats} which
are most directly related to LRNRs.

We will recall the transported subspaces, and discuss how the inverse that
appears in the composition can be avoided straightforwardly using a trick
we call the \emph{inverse-bias trick}. Starting with a given transported
subspace in its original form, we will first show that one can rewrite it
in a certain discrete form using this trick. Next, we will show that a
transported subspace in this discrete form is equivalent to a 2-layer LRNR
of comparable complexity. Finally, we will approximate the classical
solution to the scalar conservation law using a 2-layer LRNR.
The techniques developed in this section for approximating the classical
solutions will be used throughout the subsequent Sec.~\ref{sec:LRNR_approx_entropy} that addresses the
entropy solution. We refer the reader to Appendix~\ref{sec:glossary} for a
summary of notions including the transported subspaces, the inverse-bias trick,
and the LRNR approximation of the classical solution introduced in this
section. 

\subsection{The transported subspaces and the inverse-bias trick}
\label{sec:ibtrick}

We will review the definition of transported subspaces and introduce an
approximation technique which allows one to avoid calculating the
left-inverses that typically appear in method of characteristics, e.g. in
(\ref{eq:classical_superpose}, \ref{eq:ucts_generic}). 
Then we will derive a useful estimate for
transport subspaces when this technique is applied.

Throughout this section and the next, we will use the notation $\Domx$ and
$\hDomx$ to denote arbitrary open intervals in $\RR$, however in the
subsequent sections they will play roles analogous to the spatial domain
in \eqref{eq:ivp} and the extended spatial domain \eqref{eq:ext_dom} above.

\begin{definition}[Admissible coefficients]  \label{def:admit}
  Given a domain $\Domx \subset \RR$ and a set of $r$ functions $\Phi_r =
    (\phi_i)_{i=1}^r \subset P_1(\RR) \cap C(\RR)$ we say a coefficient set $A
    \subset \RR^r$ is \emph{admissible for $\Phi_r$ on $\Domx$} if
  \begin{enumerate}[label=(\roman*)]

    \item there exist constants $c_A, C_A > 0$ such that $c_A \le g' \le C_A$
          for all $g \in A \cdot \Phi_r$,

    \item $g(\RR) \supset \Domx$ for all $g \in A \cdot \Phi_r$.

  \end{enumerate}
  We will also write $\hDomx := \bigcup_{g \in A \cdot \Phi_r} g^{-1}(\Domx)$.
\end{definition}

\begin{definition}[Transported subspace] \label{def:transported_subspace}
  Suppose we are given a domain $\Domx \subset \RR$, linearly independent 
  functions
  \begin{equation} 
    \begin{aligned}
      \Phi_r & = (\phi_i)_{i=1}^r \subset P_1(\RR) \cap C(\RR),
      \\
      \Psi_r & = (\psi_i)_{i=1}^r \subset P_0(\RR) \cap BV(\RR),
    \end{aligned}
  \end{equation}
  and two coefficient sets, $A \subset \RR^r$ admissible for $\Phi_r$ on
  $\Domx$ and $B \subset \RR^r$ bounded.

  The \emph{transported subspace} is the family $H_r$ given by
  \begin{equation}
    \begin{aligned}
      \bar{H}_r & := (B \cdot \Psi_r) \circ (A \cdot \Phi_r)^{-1}
      \\
                & =
      \left\{ \bar{h} \in P_0(\Domx) \cap BV(\Domx) \ |\ 
      \bar{h} = f \circ g^{-1},\
      f \in B \cdot \Psi_r,\
      g \in A \cdot \Phi_r
      \right\}.
    \end{aligned}
  \end{equation}
  We denote by $\bar{h}(\alpha, \beta)$
  the member $\bar{h} \in \bar{H}_r$ where $(\alpha, \beta)
    \in A \times B$.

\end{definition}

The functions in $\Phi_r$ and $\Psi_r$ can be viewed as reduced basis
functions or snapshots.

Next, we introduce the inverse-bias trick. Observe the important
advantages in choosing $\Psi_r$ to be a set of piecewise constant
functions. Expressing a member $f \in B \cdot \Psi_r$ as the sum
\begin{equation} \label{eq:shallow}
  f(x) = c_0 + \sum_{k} c_k \rho(x - x_k),
\end{equation}
the transported function $h = f \circ g^{-1}$ for $g \in A \cdot
  \Phi_r$ can be rewritten
\begin{equation} \label{eq:inv-bias}
  f \circ g^{-1} (x)
  = 
  c_0 + \sum_{k} c_k \rho( g^{-1}(x) - x_k)
  = 
  c_0 + \sum_{k} c_k \rho( x - g(x_k)).
\end{equation}
In this form one readily observes that the dependence of $h$ on the function $g$
is only at the discrete set of points $\{x_k\}$, and that the inverse $g^{-1}$
is replaced by the forward map $g$. 
This follows from the simple fact that the monotonicity of $g$ implies
\begin{equation} \label{eq:inv-bias-one}
  \begin{aligned}
    \rho(g^{-1}(x) - x_k)
     & =
    \begin{cases}
      1 & \text{ if } g^{-1}(x) \ge x_k,
      \\
      0 & \text{ if } g^{-1}(x) < x_k,
    \end{cases}
    \\                   & =
                        \begin{cases}
      1 & \text{ if } x \ge g(x_k),
      \\
      0 & \text{ if } x < g(x_k),
    \end{cases}
    \\                   & =
                        \rho( x - g(x_k)).
  \end{aligned}
\end{equation}
We refer to this technique of removing the inverse applied to $g$ by instead
applying the forward map $g$ on the bias terms, the \emph{inverse-bias trick}. 

Careful consideration in Lem.~\ref{lem:compose_error} below shows that the inverse bias trick
applies to non-decreasing functions that have a left-inverse but not necessarily an inverse.
While in this paper analytical expression for $g$ originate from hyperbolic
conservation laws and are thus left-invertible, we do not require invertibility
of their numerical approximations $g_\eps$. Indeed, since the right hand
side of the inverse bias trick is independent of the inverse, we have
$\rho(g^{-1}(x) - x_k) = \rho(x - g(x_k) \approx \rho(x - g_\eps(x))$ for
which we can easily control the error without assuming invertibility of $g_\eps$, as
e.g. in Lem.~\ref{lem:compose_error}. 
In principle, one can further apply the right hand side to
arbitrary non-invertible functions $g$ such as multi-valued characteristics at
shocks, but this would invalidate the left hand side and with it the original meaning of
the identity. For such cases, we require additional techinques, as we discuss in
Sec.~\ref{sec:LRNR_approx_entropy}.
Note that
this issue of invertibility has been a source of difficulties discussed in
related approaches \cite{Hesthaven2016,CagniartMadayStamm2019,taddei2020}.

In the following lemma the approximating function is 
continuous and piecewise linear; therein
the jump function $\rho$ is replaced by a continuous piecewise
linear approximation $\rho_\eps$ defined as 
\begin{equation} \label{eq:rho-eps}
  \rho_\eps(x)
  :=
  \frac{1}{\eps}
  \left[
    \sigma(x + \frac{\eps}{2})
    -
    \sigma(x - \frac{\eps}{2})
    \right].
\end{equation}
For a finite set $\cX \subset \hDomx$ and $g: \hDomx \to \RR$, we will make use
of the discrete norm $\Norm{g}{L^\infty(\cX)} := \max_{x_i \in
    \cX}\abs{g(x_i)}$.

\begin{lemma} \label{lem:compose_error}
  Let $f: \RR \to \RR$ be of bounded variation and supported in
  $\widehat{\Dom}_x$, $g: \RR \to \RR$ lying in $W^{1, \infty}(\RR)$
  non-decreasing, $g'$ is compactly supported in $\widehat{\Dom}_x$ and $\rg{(g)} \supset \Domx$. 
  Suppose $f_\eps$ and $g_\eps$ are continuous piecewise linear functions, and
  $f_\eps$ is of the form $f_\eps = \sum_{i=1}^N c_i \rho_\eps(x - x_i)$ for
  some  grid points $\cX = \{x_i\}_{i=1}^N \subset \hDomx$ such that
  $\SemiNormlr{f_\eps}{TV(\hDomx)} \le \SemiNormlr{f}{TV(\hDomx)}$. 

  Define the function $f_\eps \circ^+ g_\eps \in C(\Domx) \cap P_1(\Domx)$
  as
  \begin{equation} \label{eq:cts_ibtrick}
    f_\eps \circ^+ g_\eps (x)
    :=
    \sum_{i=1}^N c_i \rho_\eps\left(x - g_\eps(x_i)\right).
  \end{equation}
  Then we have for parameter $\eps > 0$ sufficiently small
  depending on the grid width,
  \begin{equation}
    \begin{aligned}
       & \Norm{f \circ g^+ - f_\eps \circ^+ g_\eps}{L^1(\Domx)}
      \\
       & \le
      \Norm{g'}{L^\infty (\hDomx)} \Norm{f - f_\eps}{L^1(\hDomx)}
      \\
       & +
      \SemiNormlr{f}{TV(\hDomx)} \left[
        \left( 1 + \Norm{g'}{L^\infty(\hDomx)} \right)
        \Norm{\rho - \rho_\eps}{L^1(\RR)}
        + \Norm{g - g_\eps}{L^\infty(\cX)}
        \right].
    \end{aligned}
  \end{equation}
  In the special case $f \in P_0(\hDomx) \cap BV(\hDomx)$ the inequality reduces to
  \begin{equation}
    \begin{aligned}
       & \Norm{f \circ g^+ - f_\eps \circ^+ g_\eps}{L^1(\Domx)}
      \\
       & \le
      \SemiNormlr{f}{TV(\hDomx)} \left(
      \Norm{\rho - \rho_\eps}{L^1(\RR)}
      + \Norm{g - g_\eps}{L^\infty(\cX)}
      \right).
    \end{aligned}
  \end{equation}
\end{lemma}

\begin{proof} See Appendix~\ref{proof:lem:compose_error}. \end{proof}

Although the inverse-bias trick might appear to apply solely in the single
dimensional setting, it extends to multiple dimensions when one views the
neural network approximation as a superposition of ridge functions, for instance
via the Radon transform
\cite{Natterer,Helgason2011,Bonneel15slice,Rim18split,Rim18mr}. Close
connections have been made to neural networks, e.g. see \cite{CD89,
  Ongie2020A,Unser2023}.

\subsection{LRNR approximation of transported subspaces} \label{sec:ts}

In this section, we will show that a 2-layer version of LRNR, a simple
special case of the general definition of LRNRs introduced in
Sec.~\ref{sec:lrnr}, is able to approximate any transported subspace. The rank
of the approximating LRNR is equivalent to the dimension of the
transported subspace.

\begin{example}[A 2-layer version of LRNR] \label{expl:2layer_LRNR}
  Suppose we are given 
  \begin{enumerate}[label=(\roman*)]

    \item a set of vectors
          \begin{equation}
            \begin{aligned}
              \UU^1_r & = (U^1_i)_{i=1}^r \subset \RR^M, \,\,\,
              \UU^2_r = (U^2_i)_{i=1}^r \subset \RR^M, \,\,\,
              \VV^1_r = (V^1_i)_{i=1}^r \subset \RR^M,
              \\
              \VV^2_r & = (V^2_i)_{i=1}^r \subset \RR,
            \end{aligned}
          \end{equation}
          so that the dimensions are at most $r \in \NN$, that is,
          \begin{equation}
            \dim \UU^\ell_r,\, \dim \VV^\ell_r \le r,
            \quad \text{ for } \ell = 1, 2,
          \end{equation}

    \item coefficient sets $C = (C_1, C_2), D = (D_1, D_2)$ both in $\RR^{2r}$.

  \end{enumerate}
  Then let us denote,
  \begin{equation}
    \WW_r := (C_1 \cdot \UU^1_r ) \times (C_2 \cdot \UU^2_r ),
    \quad
    \BB_r := (D_1 \cdot \VV^1_r ) \times (D_2 \cdot \UU^2_r ).
  \end{equation}
  The corresponding LRNR $H_r$ of \emph{rank} r, \emph{d.o.f} $K$, \emph{width}
  $M$, \emph{depth} $L=2$ is given as the family of functions
  \begin{equation} \label{eq:2layer_LRNR}
    H_r
    :=
    \left\{
    h \in P_1(\RR) \cap C(\RR)
    \ \left\lvert \ 
    \begin{aligned}
       & h(x) = w_2 \cdot \sigma ( w_1 x + b_1) + b_2,\ \\
       & (w_1, w_2) \in \WW_{r},\ (b_1, b_2) \in \BB_r
    \end{aligned}
    \right.
    \right\}.
  \end{equation}
  We will denote by $h(\gamma, \theta)$ the member $h \in
    H_r$ whose coefficients are $(\gamma, \theta) \in C \times D$.
\end{example}

We now show that every transported subspace $\bar{H}_{r-1}$ can be
uniformly approximated by a LRNR $H_{r}$ with comparable rank. LRNR needs
one more rank due to a technicality in representing the jump function
using \eqref{eq:rho-eps}.

\begin{theorem} \label{lem:shallow_approx}
  For each transported subspace $\bar{H}_{r-1}$ with 
  basis functions in $\Psi_{r-1}$ taking
  on the form
  \begin{equation}
    \psi_i (x) = \sum_{j=1}^{N_i} c_{ij} \rho(x - x_{ij}),
    \quad
    i = 1, ... , r,
  \end{equation}
  a pair of coefficient sets $(A, B)$, and a given parameter $\eps > 0$, 
  there exists a
  LRNR $H_{r}$ that has d.o.f $K \lesssim \sum_{i=1}^rN_i$, 
  width $M
    \sim K$, and depth $L=2$, along with coefficient sets $C, D$ and an affine map
  $\mu: A \times B \to C \times D$
  mapping
  $\mu(\alpha, \beta) = (\gamma, \theta),$
  such that 
  the members $\bar{h}(\alpha, \beta) \in \bar{H}_{r-1}$ and $h
    (\gamma, \theta ) \in H_{r}$ satisfy
  \begin{equation}
    \Norm{ \bar{h} - h }{L^1(\Domx)}
    \le
    \left( \sum_i \abs{\psi_i}_{TV(\hDomx)} \right)
    \left( \sup_{\beta \in B} \Norm{\beta}{\infty} \right)
    \Norm{\rho - \rho_\eps}{L^1(\Domx)}.
  \end{equation}
\end{theorem}

\begin{proof} See Appendix~\ref{proof:lem:shallow_approx}. \end{proof}

The parameter error can be made arbitrarily small by increasing the width
$K$, so $\bar{H}_{r-1}$ is in the $L^1$ closure of $H_{r}$. Note also that
the total variation norm can be relaxed to a Besov norm, although we will
not pursue the details here.

\subsection{LRNR approximation of classical solutions}
Now that we have shown that there are LRNR approximations of transported subspaces
(Thm.~\ref{lem:shallow_approx}), it is natural to ask whether the classical solution
can be approximated uniformly well by a LRNR. We show that this is indeed the case.

\begin{theorem} \label{thm:lrnr_classical}
  For the classical solution 
  \begin{equation}
    u(t) = u_0 \circ X^{-1}(\cdot, t), 
    \quad
    t \in \Domt,
    \quad
    u_0 \in C^1(\Domx) \cap \cU,
  \end{equation}
  there is a LRNR $H_r$ with rank $r = 3$, d.o.f $K$, width $M \sim K$, and depth 
  $L=2$ 
  such that, for coefficients 
  $\gamma(t), \theta(t)$ that
  are linear or constant functions of $t$, 
  the member $h(t) = h(\gamma(t), \theta(t)) \in H_r$
  achieves the uniform error 
  \begin{equation} \label{eq:LRNR_classical}
    \Norm{ u(t) - h(t) }{L^1(\Domx)}
    \lesssim
    \frac{1}{K}
    \abs{u_0}_{TV(\Domx)}
    \left(
    1 + T \Norm{F''}{L^\infty(u_0(\Domx))} \Norm{u_0'}{L^\infty(\Domx)}
    \right).
  \end{equation}
\end{theorem}

\begin{proof} See Appendix~\ref{proof:thm:lrnr_classical}. \end{proof}

We remark here that the classical norm $\Norm{u_0'}{L^\infty(\Domx)}$ can
be avoided, as will become apparent in the results in the next section
(Thm.~\ref{thm:lrnr_entropy}). Its presence here is attributable to the
separate approximation of the initial condition $u_0$ and the
characteristic curves $X$, without taking into consideration the
simplifications arising in their composition. Still, this result uses a
straightforward approximation and demonstrates that, given sufficient
width, the intrinsic low dimension can be exploited by LRNR.

We briefly emphasize again that while the transported subspace requires
that its coefficient sets be admissible, the LRNR does not
require similar
constraints. The representation makes sense for any coefficients $\gamma$
and $\theta$, and this is a consequence of the inverse-bias trick.

\section{LRNR approximation of entropy solutions}
\label{sec:LRNR_approx_entropy}

Having constructed an efficient LRNR approximation of the classical solution,
we turn our attention to the LRNR approximation of the entropy solution.
As briefly discussed in the beginning of Sec.~\ref{sec:LRNR_classical}, there are
important differences between the classical solution and the entropy solution.
Let us explain how these differences introduce significant challenges in
constructing the LRNR approximation.

Recall from Sec.~\ref{sec:entropy_theory}
that we have introduced two new representations of the entropy solution,
namely, one using the rarefied characteristics $\hcharX$
\eqref{eq:u0Yp} and another using the relief characteristics $\chkcharX$
\eqref{eq:ucts}:
\begin{equation} \label{eq:class_rare_relief}
  \begin{aligned}
  &u_0 \circ X^{-1} (x, t) & &\text{ classical characteristics,}
  \\
  &\hat{u}_0 \circ \hcharX^{+} (x, t)
  & &\text{ rarefied characteristics,}
  \\
  &\hat{u}_0 \circ \chkcharX^{+} (x, t)
  & &\text{ relief characteristics.}
  \\
  &\quad =\hat{u}_0 \circ ( \hId + t (F' \circ \hat{u}_0) +
                     C_x \rho \circ (t - \lambda) )^+
  \end{aligned}
\end{equation}
The rarefied characteristics $\hcharX$ is a direct generalization of the
classical chracteristics that allows representing shocks and rarefaction waves.
It is monotone, thus the inverse-bias trick can still be applied
directly to handle the left-inverse $(\cdot)^+$. However, as we discuss below,
the rarefied characteristics are not in a compositional low rank form that 
is compatible with an efficient representation by LRNRs.
On the other hand, 
the relief chracteristics $\chkcharX$ do have a compositional low rank
structure (as discussed in Sec.~\ref{sec:lrnr}) that is compatible
with LRNRs. However, it is not monotone
and hence the left-inverse cannot be
removed using the inverse-bias trick directly and an additional layer of
approximation is required (see paragraphs following \eqref{eq:inv-bias}).

Therefore, our goal here is to devise a LRNR approximation that agrees with 
the rarefied
characteristics $\hcharX$, while simultaneously exploiting the compositional
low-rank structure of the relief characteristics $\chkcharX$. We build a
construction by decomposing the rarefied characteristics into sum of local-in-time
contributions, and introduce a LRNR structure per each summand via the same
technique used to for relief characteristics.
This allows us to take advantage of the compositional structure while
not necessarily maintaining monotonicity.

In this section we will address these challenges involved and construct an efficient LRNR approximation of the
entropy solution. We first provide a broad outline in
Sec.~\ref{sec:LRNR_entropy_overview}, 
then introduce various discretizations that
gradually reveal the neural-network-like structure in
Sec.~\ref{sec:timelayer} and Sec.~\ref{sec:timelayer2}, to finally work towards the main
Thm.~\ref{thm:lrnr_entropy} in Sec.~\ref{sec:LRNR_entropy}. Throughout, we
refer the reader to Appendix~\ref{sec:glossary} for a glossary of various
approximations introduced in this section. 

\subsection{Outline of approach}
\label{sec:LRNR_entropy_overview}

\begin{figure}
  \centering
  \includegraphics[width=0.9\textwidth]{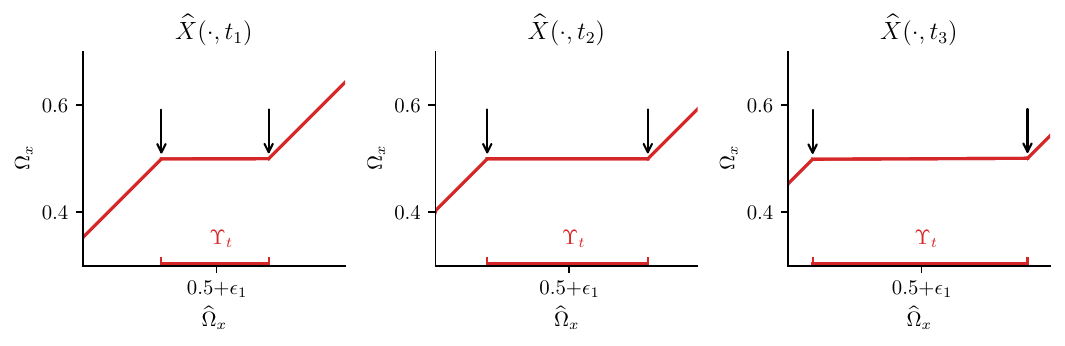}
  \caption{The time evolution of the rarefied characteristics
  $\hcharX(\cdot, t)$ for the stationary shock example in Fig.~\ref{fig:stationary},
  zoomed into the region relevant to shock propagation.
  The constant interval (marked by $\shockset$) expands over time $0 < t_1 < t_2
  < t_3$. The two kinks at the endpoints (marked by two arrows) of this interval
  travel in opposite directions over time.}
  \label{fig:stationary_movie}
\end{figure}

In this section, we will motivate and outline our general approximation
approach.  We begin by 
detailing the technical challenges mentioned in the introduction
 of this section.
In light of the inverse-bias trick in Sec.~\ref{sec:ibtrick}, it might appear
that the approximation of the entropy solution is now a simple matter of
approximating the rarefied characteristics $\hcharX$
\eqref{eq:rarefied_charcurves}. Unfortunately, the rarefied characteristics
manifold given by
\begin{equation}
  \cM_{\hcharX} := \{\hcharX(\cdot, t) : t \in \Domt\},
\end{equation}
itself has a slowly decaying Kolmogorov $n$-width in the presence of shocks.
For example, recall the stationary shock example from
Fig.~\ref{fig:stationary}. Observe how the rarefied characteristics
$\hcharX(\cdot, t)$ evolves over time, as illustrated in
Fig.~\ref{fig:stationary_movie}. The constant region of $\hcharX(\cdot, t)$
describing the shock expand with time, according to the Rankine-Hugoniot 
condition, and this leads to kinks that travel
across the domain $\hDomx$. As a result, the $n$-width in $L^2$-norm must be
$\gtrsim n^{-3/2}$
\cite{Welper2017,Ohlberger16,Greif19,rim2020depth}. 
This is an obstruction to the strategy of finding a compositional
reduced model directly, by using e.g. transported subspaces
(Def.~\ref{def:transported_subspace}). 
In other words, there is now a slow decay in the $n$-width of
\emph{both} the solution manifold \emph{and} the rarefied characteristics
manifold.\footnote{To provide a succinct
  analogy: The classical characteristics are to the solution
    manifold, what the shock time function is to the rarefied
    characteristics manifold.}

One natural question to ask is whether it is possible to 
construct compositional reduced models yet again, this
time for the manifold of rarefied characteristics $\cM_{\hcharX}$. A 
low-rank structure is apparent in the relief
characteristics $\chkcharX$ \eqref{eq:relief}, in the superpositions of a
few basis functions, as discussed in Sec.~\ref{sec:lrnr};
see the diagram in \eqref{eq:low_rank_upto_compositions}. Still, 
we need care in introducing the shock-related jump term $\rho(\lambda(x) -
  t)$ that appears in the inner-most composition \eqref{eq:relief}. In the
relief characteristics, this term is used in conjunction with the
left-inverse to replace the portions of the rarefied characteristic curves
with constant regions. LRNRs are defined using ReLU activations so they
are continuous piecewise linear functions; as such, the jump discontinuity
must be approximated carefully so that the low-rank compositional structure
is preserved.
We propose a LRNR approximation that satisfies these requirements, outlined as follows. 

\begin{itemize}[itemsep=2pt, topsep=4pt]
  \item We start with a continuous
piecewise linear approximation of the rarefied characteristics with
multivalued inverse \eqref{eq:hcharX0},
\begin{equation} \label{eq:hX0_rhoeps}
  \hcharX_0 (\cdot, t)
  \approx
  \sum_k c_k (t) \rho_\eps
  \left(
  \,\cdot\, - x_k
  \right).
\end{equation}
This approximation so far does not properly handle shocks.
\item As discussed above, the LRNR approximation of shock propagation must
be done
carefully to preserve the compositional low rankness of the overall 
approximation and avoid the slow decay in the Kolmogorov width.
We represent shocks (constant regions) by removing certain
variations in 
the characteristics, emulating the relief characteristics while preserving 
continuous piecewise linearity and monotonicity. 
Roughly speaking, we achieve this by introducing
a time-layered approximation so that
slices of $\hcharX_0(\cdot ,t)$ are translated outside 
of the domain. 

To maintain uniform error between $\hcharX$ and this approximation, it is
necessary that the slice that is being translated to the right 
does so quickly and is small in amplitude.
A key observation that makes this possible is that,
in the approximation \eqref{eq:hX0_rhoeps},
there is a freedom in constructing the coefficients $(c_k)$.
Without adding to the number of coefficients, one may restructure 
the coefficients so that the approximation is rewritten as a sum of
the slices $(\eta_k)$. With some simplifications, one writes
\begin{equation} 
  \hcharX_0 (\cdot, t)
  \approx
  \hId(\cdot) + t v(\cdot)
  +
  \sum_k  t\, \eta_k \left( \,\cdot\, \right),
  \qquad
  \Norm{t \eta_k }{L^\infty(\hDomx)} \lesssim \eps.
\end{equation}
Note that this approximates the multi-valued inverse and
does not properly represent shocks.
\item  We modify the above approximation to represent shocks,
by an appropriate removal of slice terms to mimick the expanding constant 
regions of $\hcharX$. That is, we write
\begin{equation} 
  \hcharX (\cdot, t)
  \approx
  \barwedge(\cdot, t)
  \approx
  \hId(\cdot)+ t v(\cdot)
  +
  \sum_{\mathclap{\substack{\text{partial sum}\\ \text{over }k}}} 
  t\, \eta_k
  \left( \,\cdot\, \right),
  \qquad
  \Norm{t \eta_k }{L^\infty(\hDomx)} \lesssim \eps,
\end{equation}
where $\barwedge$ denotes a piecewise linear approximation of
$\hcharX$ we will construct in the following section.
Now, we have an appropriate approximation of the rarefied characteristics, 
thereby representing shocks properly. However, the partial summation over $k$
is not a compositional operation, so it is not yet amenable
to efficient LRNR approximations.

\item Now each slice $\eta_k$ can be removed from the domain
using the inverse-bias trick devised for \eqref{eq:relief}: This amounts to 
inserting 
an approximation of the term $\rho(t - \lambda(x))$ into the bias,
\begin{equation} \label{eq:hcharX_approx}
  \hcharX (\cdot, t)
   \approx
  \doublebarwedge(\cdot, t)
  \approx
  \hId(\cdot)
  + t v(\cdot)
  +
  \sum_k   
  t
    \eta_k \left(
  \,\cdot\, - C_x \rho_\eps(t - \lambda(x_k)) \right),
  \quad
  x_k \in \supp \eta_k,
\end{equation}
with $C_x$  sufficiently large \eqref{eq:relief}. Finally, the
$\doublebarwedge$
is a continuous piecewise linear approximation that is efficiently
representable in the form of a LRNR that we will construct below.
This approximation is not monotone, but it is still well-defined where the approximation
is not monotone and is compatible with the 
inverse-bias trick, as an approximation to $\hcharX$
(Sec.~\ref{sec:ibtrick}  and Lem.~\ref{lem:compose_error}) . 
 Note that this
approximation has the compositional structure of relief
characteristics; the bias term
$C_x \rho_\eps(t - \lambda(x_k))$ plays the role of the term
$C_x \rho (t - \lambda(x))$ in $\chkcharX$ \eqref{eq:class_rare_relief}.
Whenever $|t - \lambda(x_k)| \lesssim \eps$, i.e. approximately the time
when the characteristic 
curve emanating from $x_k$ has entered into the shock, 
a portion of the characteristic
information $\eta_k$
is moved outside of the domain $\hDomx$ with speed $\sim 1/\eps$.
\end{itemize}

We present our construction as follows. In Sec.~\ref{sec:timelayer}, we derive
a time-layered approximation of rarefied characteristics $\hcharX$ that is
valid for short time intervals, then in Sec.~\ref{sec:timelayer2} we construct
a discretization of the form \eqref{eq:hcharX_approx}, then finally in
Sec.~\ref{sec:LRNR_entropy} we present the LRNR approximation result.

In the following sections, we will make use of a few special symbols to
denote functions, due to the intuitive relation between their shape and
the referenced function: $\vdash, \dashv, \barwedge$, and
$\doublebarwedge$ are
standard symbols used throughout the mathematics literature for various
purposes.\footnote{These symbols can be interpreted as Hangul
  letters, which makes them pronounceable as \emph{a}, \emph{eo},
  \emph{jieut}, and \emph{chieut}, respectively \cite{Taylor1980}.}

We close this outline with the remark that the use of a threshold function
in \eqref{eq:hcharX_approx} shares a similar theme with TSI
\cite{Welper2017}, transport reversal \cite{rim17reversal}, or FTR
\cite{Krah2023}.

\subsection{A time-layered approximation of the rarefied characteristics}
\label{sec:timelayer}

We first introduce an approximation of $\hcharX(x, t)$
\eqref{eq:rarefied_charcurves} that is accurate for a given time interval.
This approximation will be uniformly accurate, with the error controlled
by the number of degrees of freedom depending on the complexity of $u_0$
and the flux function $F$.

To simplify our discussion, let us define two functions we call \emph{left
  shock endpoint} and \emph{right shock endpoint} of a shock containing the
point $x$ at time $t$,
\begin{equation}
  \skleft,\, \skright : 
  \{ (x,t) \mid x \in \shockset, t \in \Domt \} \to \hDomx,
\end{equation}
which are given respectively by
\begin{equation} \label{eq:sklr}
  \skleft (x, t)
  := 
  \inf_z
  \left\{
  [z, x] \in \shockset
  \right\},
  \Aand
  \skright (x, t)
  := 
  \sup_z
  \left\{
  [x, z] \in \shockset
  \right\}.
\end{equation}
Note that these endpoints can represent arbitrary number of shocks;
the ranges $\skleft(\shocksett{t}, t)$ and $\skright(\shocksett{t}, t)$
yield left and right endpoints of \emph{all} shocks that have formed
up to time $t$.  
For the points $x \in \Upsilon_t$, that is, the points in $\hDomx$
whose 
subsequent characteristic curves go into the shock at time $t$, the
function values $\skleft (x,t)$ and $\skright (x,t)$ are the left and right most
end points of the maximal interval containing the point that are in the shock at
the time.

Since $\hcharX(\cdot, t)$ is continuous and it is constant in the interval
$[\skleft(x,t), \skright(x,t)]$,
\begin{equation}
  \hcharX(\skleft(x, t), t)
  =
  \hcharX(\skright(x, t), t)
  =
  \hcharX(z, t),
  \quad
  z \in [\skleft(x, t), \skright (x,t)].
\end{equation}
Furthermore, since we assumed our initial condition was in
$\overline{\cU}$ \eqref{eq:u0space}, and our problem is homogeneous
\eqref{eq:ivp},
\begin{equation}
  \# \left( \skleft (\hDomx, t) \right)
  =
  \# \left( \skright (\hDomx, t) \right)
  \le
  \# \left( \{ x \in \hDomx: u_0 (x_-) \ne u_0(x_+) \} \right),
\end{equation}
where the RHS is the number of jumps in the initial condition $u_0$,
and $\# (\cdot)$ denotes the cardinality of the set.  
For such
an initial condition, the number of jumps does not increase \cite{Ser99}.

For a fixed $x \in \hDomx$, the shock endpoints $\skleft(x, \cdot)$ and
$\skright(x, \cdot)$ are strictly decreasing and increasing with $t$,
respectively, that is,
\begin{equation}
  \text{ if }
  \quad
  0 \le t_1 < t_2 \le T
  \Tthen
  \skleft (x, t_2) < \skleft (x, t_1)
  \, \text{ and } \,
  \skright (x, t_1) < \skright (x, t_2),
\end{equation}
since the characteristic curves always enter into the shock, not vice versa.

Due to the differential equations that yield this evolution
\eqref{eq:rarefied_charcurves}, the endpoints $\skleft, \skright$ are
piecewise continuously differentiable with respect to $t$, except at
finitely many points in $\Domt$.

Let $k \in \NN$, $t_k \in \Domt$ with $t_k \le t_{k+1}$, then define the
indexed functions we call \emph{time-layered approximations} $\Xi_k \in
  C(\hDomx \times \Domt)$ of $\hcharX$,
\begin{equation} \label{eq:Xik}
  \Xi_k (x, t)
  :=
  \hId(x) + t \, v_k(x),
\end{equation}
where $v_k$ is defined as the continuous function
\begin{equation} \label{eq:vk}
  v_k(x)
  :=
  \left\{
  \begin{aligned}
    (F' \circ \hat{u}_0) (x)
     &  & 
    \text{ if } \lambda(x) > t_k,
    \\
    (F' \circ \hat{u}_0) (\skleft(x, \, t_k))
     &  & 
    \text{ if } \lambda(x) \le t_k.
  \end{aligned}
  \right.
\end{equation}
Note that $v_k(x)$ is constant in each connected component of the domain $\shocksett{t_k} = \{ x \in
  \hDomx : \lambda(x) \le t_k \}$, and that $v_k$ can be defined using $\skright$
in place of $\skleft$. The function $\Xi_k$ satisfies the following estimates.

\begin{lemma} \label{lem:Zk}
  For $t \in [t_k, t_{k+1}]$ we have the estimates,
  \begin{align}
    \Norm{\hcharX(\cdot, t) - \Xi_k (\cdot, t)}{L^\infty(\hDomx)}
     & \le
    \SemiNormlr{u_0}{TV(\Domx)}
    \Norm{F''}{L^\infty(u_0(\Domx))} \abs{t_{k+1} - t_k},
    \label{eq:hXZk_error}
    \\
    \SemiNormlr{\Xi_k (\cdot, t)}{TV(\hDomx)}
     & 
    \le
    1 
    + 
    T \SemiNormlr{u_0}{TV(\Domx)} \Norm{F''}{L^\infty(u_0(\Domx))}.
    \label{eq:ZkTV}
  \end{align}
\end{lemma}

\begin{proof} See Appendix~\ref{proof:lem:Zk}. \end{proof}

\begin{figure}
  \includegraphics[width=0.7\textwidth]{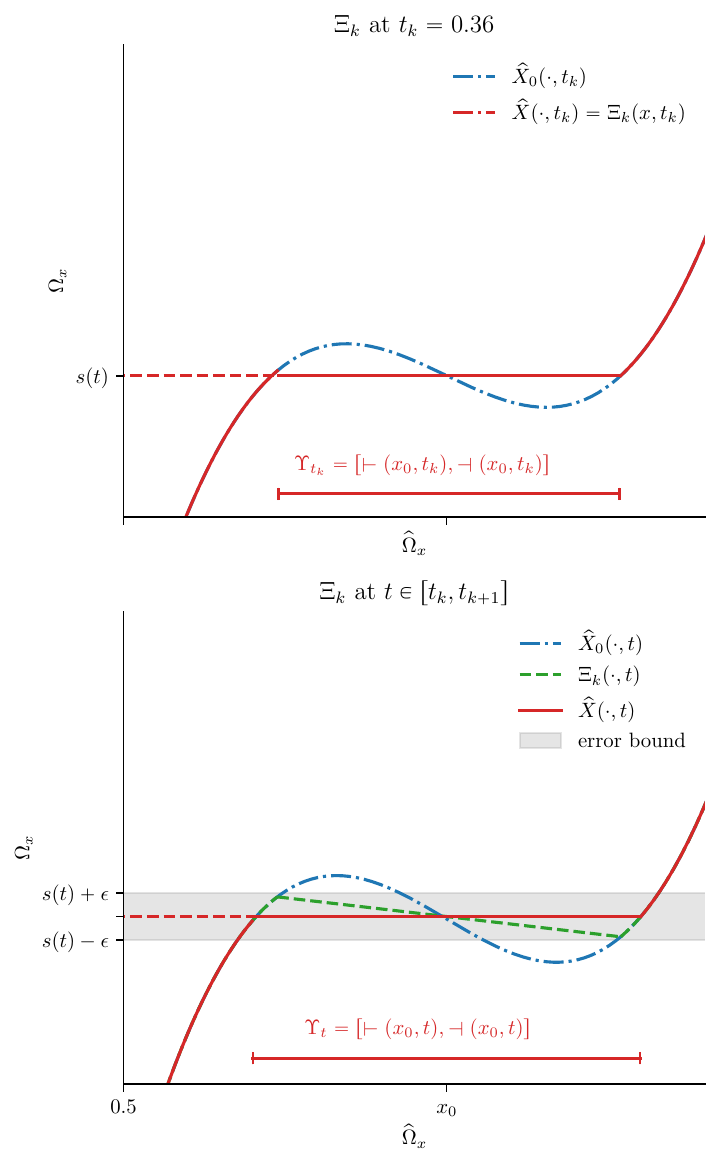}
  \caption{An illustration of time-layered approximation $\Xi_k$ \eqref{eq:Xik} 
  immediately after shock formation, compared with the rarefied characteristics with multivalued inverse $\hcharX_0$ and the rarefied characteristics $\hcharX$. At uniform grid times $t_k$ \eqref{eq:tunif} the $\Xi_k(\cdot, t_k) = \hcharX(\cdot, t_k)$, however, as the time is evolved forward $\Xi_k(\cdot, t)$ is different from $\hcharX(\cdot, t_k)$ in general, and the difference is within the $\epsilon$ error bound.
    }
  \label{fig:timelayer}
\end{figure}

\begin{figure}
  \centering
  \includegraphics[width=1.0\textwidth]{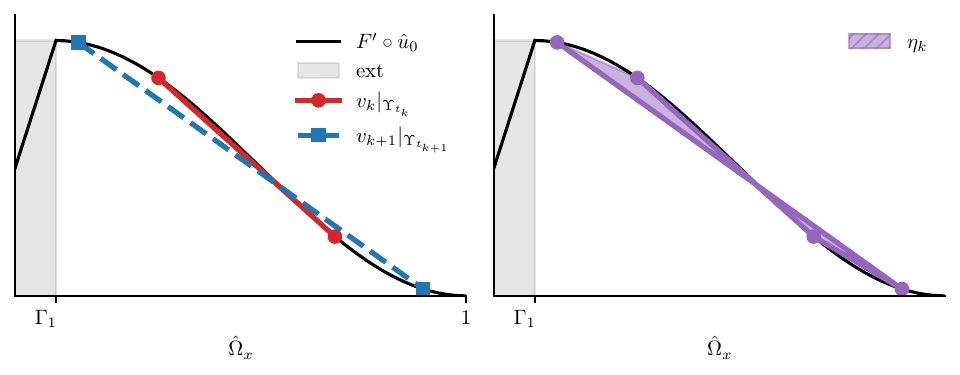}
  \caption{An illustration of the interpolant $\eta_k$ \eqref{eq:etak} along
      with the functions $v_k$ and $v_{k+1}$ \eqref{eq:vk}. The thin slice $\eta_k$
      is plotted to the right.}
  \label{fig:etak}
\end{figure}

Next, suppose we let $(t_k)_{k=1}^K$ be a uniform grid over time domain $\Domt$, 
\begin{equation} \label{eq:tunif}
  t_k := (k-1) \frac{T}{K},
\end{equation}
then we have our time-layered approximation $\Xi_k$ which satisfies
the uniform estimate
\begin{equation}\label{eq:hchar_Xi_err}
  \Norm{\hcharX(\cdot, t) - \Xi_k(\cdot, t)}{L^\infty(\hDomx)}
  \le
  \frac{T}{K}
  \SemiNormlr{u_0}{TV(\Domx)} \Norm{F''}{L^\infty(u_0(\Domx))}.
\end{equation}
See Fig.~\ref{fig:timelayer} for an illustration of this approximation.

\subsection{The spatial discretization of time-layered approximation}
\label{sec:timelayer2}

Now we introduce the appropriate spatial discretization of the
time-layered approximations $\Xi_k$ \eqref{eq:Xik} that we have constructed
in the previous section.

Let us introduce a notation for a constant that will appear repeatedly,
\begin{equation}
  \Const
  :=
  T \SemiNormlr{u_0}{TV(\Domx)} \Norm{F''}{L^\infty(u_0(\Domx))}.
\end{equation}
In the following, we will make use of 
the estimates \eqref{eq:hXZk_error} and \eqref{eq:ZkTV}
in Lem.~\ref{lem:Zk}, and the uniform time grid $t_k$
with grid size $K$ \eqref{eq:tunif}.

Because of the TV bound \eqref{eq:ZkTV}, there is an adaptive spatial grid
$(z_i)_{i=1}^K$ over $\hDomx$ on which we have  
continuous piecewise linear approximations
$\barwedge_k(\cdot, t) \in C(\hDomx \times \Domt) \cap P_1(\hDomx \times
  \Domt)$ of $\Xi_k(\cdot, t)$ satisfying
\begin{equation} \label{eq:Zk_error}
  \Norm{\Xi_k(\cdot, t) - \barwedge_k(\cdot, t)}{L^\infty(\hDomx)}
  \le
  \frac{1}{K}
  \left( 1 + \Const \right),\quad \text{ for } k= 1, ...\,, K-1.
\end{equation}
We refer the reader to \cite{Devore1998} for the details of the estimate and the
grid construction.  Note here that we have chosen the size of the adaptive grid
to equal the size $K$ of the time grid.  Furthermore, 
we choose $\barwedge_k$
so that it is an interpolant that satisfies 
\begin{equation}\label{eq:Zk_interp}
  \barwedge_k( z_i, t_k) = 
  \left\{
  \begin{aligned}
    \hcharX_0(z_i, t_k)
     & 
    \Ffor z_i \in \hDomx \setminus \shocksett{t_k},
    \\
    \hcharX ( z_i, t_k )
     & 
    \Ffor z_i \in \shocksett{t_k}.
  \end{aligned}
  \right.
\end{equation}

We have now obtained the interpolants $(\barwedge_k)$, and next we wish to
exploit the fact that these interpolants are closely related to each other
across the index $k$. They interpolate $(\Xi_k)$ each in the form
\eqref{eq:Xik} containing the functions $(v_k)$. Observe that these
functions approximate $F' \circ \hat{u}_0$ in $\hDomx \smallsetminus
  \shocksett{t_k}$. Since $(v_k)$ each approximate the same spatial
function, it is possible to construct an approximation of $(\Xi_k)$ for
all $k$ simply by reorganizing the point values of $(\barwedge_k)$ by
adding or removing the differences $(v_{k+1} - v_{k})$. We do this by
introducing a continuous piecewise linear interpolant of these
differences. 

For each index $k$, let $\eta_k \in C(\hDomx) \cap P_1(\hDomx) $ be an
approximation of $v_{k+1} - v_k$ on the grid $(z_i)$, specifically,
\begin{equation} \label{eq:etak}
  \eta_k (z_i)
  =
  \frac{1}{t_{k+1}}
  (\barwedge_{k+1} (z_i, t_k) - \hId (z_i) )
  -
  \frac{1}{t_k}
  (\barwedge_{k} (z_i, t_k) - \hId (z_i))
  \quad
  i=0, 1, ...\,, K-1.
\end{equation}
Then $\eta_k$ is nonzero only at the grid
points $(z_i)_{i=1}^K$ that lie in $\shocksett{t_{k+1}} \cup \shocksett{t_k}$.
See Fig.~\ref{fig:etak} for an illustration.

A quick calculation yields for $t \in [t_k, t_{k+1}]$,
\begin{equation} \label{eq:etakinfnorm}
  \begin{aligned}
     & \Norm{t\eta_k}{L^\infty(\hDomx)}
    \\                    & \le
                         \Norm{t\eta_k - t (v_{k+1} - v_k)}{L^\infty(\hDomx)}
    +
                         \Norm{t (v_{k+1} - v_k)}{L^\infty(\hDomx)}
    \\                    & \le
                         t\Norm{\eta_k - (v_{k+1} - v_k) }{L^\infty(\hDomx)}
    +
                         \Norm{ \Xi_k - \Xi_{k+1} }{L^\infty(\hDomx)}
    \le
                         \frac{4 \Const}{K},
  \end{aligned}
\end{equation}
due to \eqref{eq:hXZk_error} and \eqref{eq:Zk_error}.

Next, let us set the function $\barwedge: \hDomx \times \Domt \to \RR$
which is continuous piecewise linear in space but only piecewise linear in
time. Let $\barwedge_K(x, t):= \hId(x) + \frac{t}{T} (\barwedge_K(x, t_K) -
\hId(x))$ and

\begin{equation} \label{eq:jieut_partial_sum}
  \barwedge(x, t)
  :=
  \barwedge_K(x, t)
  +
  t \left[ \sum_{k: t < t_k} \eta_k (x) \right]
\end{equation}
for which we have the estimate,
\begin{equation} \label{eq:barwedge_error}
  \begin{aligned}
     & \Norm{\hcharX(\cdot, t) - \barwedge(\cdot, t) }{L^\infty(\hDomx)}
    \\
     & \le
    \Norm{\hcharX(\cdot, t) - \Xi_k(\cdot, t) }{L^\infty(\hDomx)}
    \\
    & \quad
    ~~~ +
    \Norm{ \Xi_k(\cdot, t) - \barwedge_k(\cdot, t_k) }{L^\infty(\hDomx)}
    +
    \Norm{ \barwedge_k(\cdot, t_k) - \barwedge(\cdot, t)}{L^\infty(\hDomx)}
    \\
     & 
    \le
    \frac{\Const}{K}
    +
    \frac{1 + \Const}{K}
    +
    \frac{4 \Const}{K}
    \\
     & 
    \le
    \frac{1 + 6\Const}{K}.
  \end{aligned}
\end{equation}
Finally, define the fully continuous piecewise linear 
$\doublebarwedge \in C(\hDomx \times \Domt) \cap P_1(\hDomx \times
  \Domt)$,
\begin{equation} \label{eq:chieut_approx}
  \doublebarwedge (x, t)
  := 
  \barwedge_K(x,t)
  +
  t
    \sum_{k=0}^{K-1}
    \eta_k\left(
    x - C_x \sigma \left( \frac{K}{T} (t - t_k) \right) 
    \right),
\end{equation}
for a constant $C_x \ge \abs{\Domx}$. Notice how the summation indices are over all $k$: Essentially, through 
the modification $x \mapsto x - C_x \sigma \left( \frac{K}{T} (t - t_k)
\right)$ we are emulating the partial sum over $k$ in
\eqref{eq:jieut_partial_sum} by quickly translating the $\eta_k$ terms outside
of the spatial domain $\hDomx$.

We briefly digress to point out that this approximation is the
implementation of \eqref{eq:hcharX_approx}. One can show that the uniform
time grid-point $t_k$ appearing in the inner-most argument is
approximately $\lambda(z)$ for some $z \in \supp \eta_k$ up to $\eps$
error. This fact, together with \eqref{eq:etakinfnorm}, implies that
$\doublebarwedge$ is indeed approximated by the expression
\begin{equation} \label{eq:relief_approx}
  \doublebarwedge (x, t)
  \approx 
  \barwedge_K (x,t)
  +
  t
    \sum_{k=0}^{K-1} \eta_k
    \left(x 
    - C_x \sigma \left(\frac{K}{T} (t - \lambda(z_{i_k}))\right) \right),
\end{equation}
in which $z_{i_k} \in (z_i)_{i=1}^K$ is any grid point picked to be in $\supp
  \eta_k$. Noting that $\eta_k$ are piecewise linear functions that can 
be approximated by a linear combination of functions $\rho_\eps(\cdot - x_k)$
appearing in \eqref{eq:hcharX_approx}, and that $K/T \sim 1/\eps$,
one sees that $\doublebarwedge$ indeed
approximates $\hcharX$ in the manner described above.
The approximation is thus an analogue of the relief characteristics
$\chkcharX$ described above \eqref{eq:relief}; it omits the part
of the characteristics that has entered the shock, simply by moving the 
data out of the domain, using a certain thresholding of the shock-time function
in conjunction with the inverse-bias trick.
See Fig.~\ref{fig:chieut} for a depiction of how $\doublebarwedge$ approximates
the constant region representing shocks.

We will next prove a uniform error estimate for $\doublebarwedge$.

\begin{lemma} \label{lem:chieut}
  The approximation 
  $\doublebarwedge \in C(\hDomx \times \Domt) \cap P_1(\hDomx \times \Domt)$
  defined above in \eqref{eq:chieut_approx} satisfies
  \begin{equation} \label{eq:hcharX_error}
    \Norm{\hcharX(\cdot, t) - \doublebarwedge(\cdot, t)}{L^\infty(\hDomx)}
    \lesssim
    \frac{1}{K}
    \left(
    1 + T \SemiNormlr{u_0}{TV(\Domx)} \Norm{F''}{L^\infty(u_0(\Domx))}
    \right).
  \end{equation}
\end{lemma}

\begin{proof} See Appendix~\ref{proof:lem:chieut}. \end{proof}

\begin{figure}
  \centering
  \includegraphics[width=1.0\textwidth]{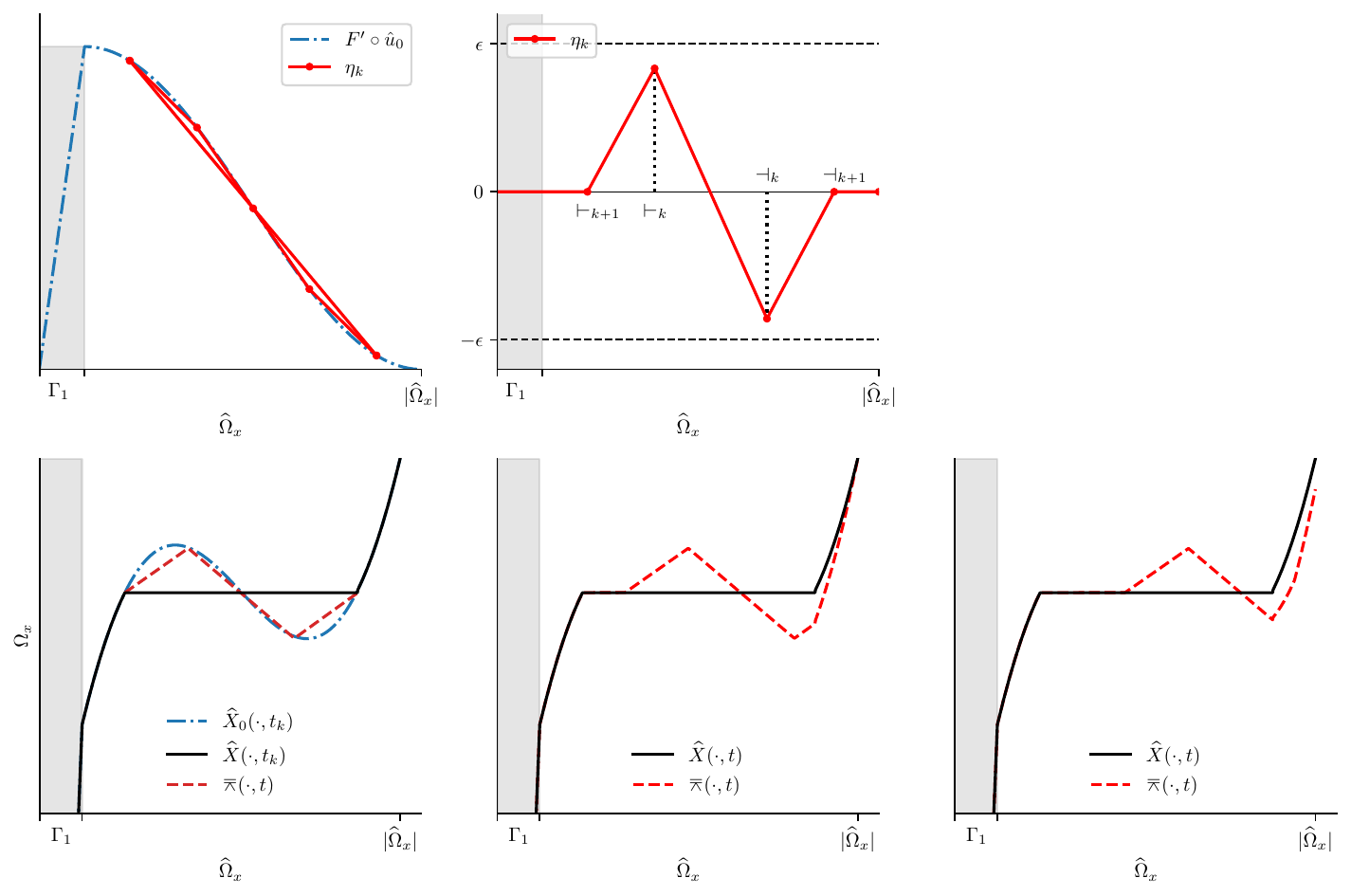}
  \caption{A diagram depicting the approximation by $\doublebarwedge(x, t)$.
    The term $F' \circ \hat{u}_0$ in the rarefied characteristics with
      multivalued inverse $\hcharX_0$ along with the interpolant $\eta_k$
      \eqref{eq:etak} are shown in the upper left plot. The interpolant $\eta_k$
      along with the left and right shock endpoints $\skleft, \skright$ at times
      $t_k$ and $t_{k+1}$ are shown in the upper center plot.
      The three lower plots show a movie illustrating the evolution of
      $\doublebarwedge(x, t)$ starting at time $t = t_k$. The thin interpolant $\eta_k$  representing a
      tiny perturbation of the characteristic curves, 
     is 
      moving out of the spatial domain to the right, 
      traveling quickly outside the domain.
  }
  \label{fig:chieut}
\end{figure}

Based on the continuous piecewise linear analogue $\doublebarwedge$ of the
rarefied characteristics, we now construct an approximation of the entropy
solution $u$. We will make use of known approximations of $u$ (see, e.g. \cite{Dafermos2010}): There is an entropy solution
  $\bar{u}$ with piecewise constant initial datum $\bar{u}_0 \in
    \overline{\cU}$ satisfying
  \begin{equation}
    \Norm{u(\cdot, t) - \bar{u}(\cdot, t)}{L^\infty(\Domx)}
    \lesssim
    \frac{1}{K} \SemiNormlr{u_0}{TV(\Domx)},
    \quad
    t \in \Domt.
  \end{equation}
  Let $\hat{u}_0$ be the extended initial condition of such $\bar{u}_0$
  constructed as in \eqref{eq:uhat}.
  Let $\hat{u}_{0\eps}$ be its approximation (see, e.g. \cite{Devore1998}) of the form
  \begin{equation} \label{eq:hatu0eps}
    \hat{u}_{0\eps}(x)
    =
    \sum_{i=1}^K \hat{u}_{0,i} \, \rho_\eps(x - x_i),
  \end{equation}
  that satisfies the error bound 
  \begin{equation}
    \Norm{\hat{u}_0 - \hat{u}_{0\eps}}{L^1(\hDomx)} \lesssim \frac{1}{K} \abs{u_0}_{TV(\Domx)}.
  \end{equation}
  Furthermore, 
  let $\doublebarwedge$ be the approximation constructed above \eqref{eq:chieut_approx}
  for the solution $\bar{u}$.

\begin{theorem} \label{thm:pwlin_entropy}
  Let $u$ be the entropy solution to the scalar conservation law
  \eqref{eq:ivp}. 
  Then for given $\hat{u}_{0\eps}$ and $\doublebarwedge$ as above
  and $\eps = 1/K$,
  \begin{equation} \label{eq:pwlin_entropy}
    \bar{h}(x, t)
    :=
    \hat{u}_{0\eps} \circ^+ \doublebarwedge (x, t)
    =
    \sum_{i=1}^K c_i \rho_\eps (x - \doublebarwedge(x_i, t))
  \end{equation}
  approximates the entropy solution $u$ with the error estimate
  \begin{equation} \label{eq:entropy_approx_error}
    \begin{aligned}
       & \Norm{u(\cdot, t) - \bar{h}(\cdot,t))}{L^1(\Dom_x)}
      \\
       & \quad \lesssim
      \frac{1}{K} \SemiNormlr{u_0}{TV(\Domx)}
      \left( 1 + \SemiNormlr{u_0}{TV(\Domx)} \right)
      \left( 1 + T \Norm{F''}{L^\infty(u_0(\Domx))} \right).
    \end{aligned}
  \end{equation}
\end{theorem}

\begin{proof} See Appendix~\ref{proof:thm:pwlin_entropy}. \end{proof}

The major issues identified in the introduction of this section
are the monotonicity and the low rank compositional structure. The final
approximation $\doublebarwedge$ of the rarefied characteristics is 
not necessarily monotone but sufficiently accurate, 
allowing the approximate version of the inverse-bias trick via
Lem.~\ref{lem:compose_error}.
Observe how in \eqref{eq:chieut_approx} one sees that it has a low rank
compositional structure reminiscent of the relief characteristics. This
will form the foundation for constructing the LRNR approximation in
Thm.~\ref{thm:lrnr_entropy} below.

The approximation error in \eqref{eq:entropy_approx_error} scales as $\sim
  K^{-1}$ and the constants depend only on the initial condition and the
flux function. is important to note that this rate implies that our
construction is not a na\"ive space-time approximation, for which the
error would scale with the d.o.f $\sim K^{-\half}$. 

Let us discuss the number of operations needed to evaluate the
approximation $\bar{h}$ \eqref{eq:pwlin_entropy} at points in the
space-time domain. At initial time, to evaluate $\doublebarwedge(x_i, t) $
for $i = 1, ...\, , K,$ one needs the evaluation of $\hId(x_i), v_K(x_i),
$ which requires $\cO(K)$ operations, so one evaluation of $h(x, 0)$
requires $\cO(K)$ work. Then, to evaluate $h(\bar{x}_p, 0)$ for a grid
$(\bar{x}_p)_{p=1}^K \subset \Domx$, $\cO(K^2)$ work is needed. Then, to
evaluate $h(\bar{x}_p, t_k)$ on the uniform time-grid $t_k$, one marches
forward in time: At each time $t_{k}$ the evaluation $t_{k+1}$ can be
obtained by sweeping through the spatial grid once, and each update costs
$\cO(1)$ for each $\bar{x}_i$, resulting in $\cO(K)$ work for each
time-update, by only adjusting one term in the sum appearing in the
definition of $\doublebarwedge$ in \eqref{eq:chieut_approx}. Thus, one can
evaluate the piecewise linear approximation $\bar{h}$
\eqref{eq:pwlin_entropy} at a space-time grid of size $\sim K^2$ using
$\cO(K^2)$ workload. This amounts to $\cO(1)$ work for per space-time grid
point.

Moreover, in this space-time approximation, the temporal approximation is
trivial, due to the explicit linear time dependence as clearly seen in
\eqref{eq:chieut_approx}.

\subsection{The LRNR approximation of compositional form}
\label{sec:LRNR_entropy}

We now prove that a LRNR construction can implement the approximation from
the previous section.

\begin{theorem}  \label{thm:lrnr_entropy}
  There is a LRNR $H_r$ with rank $r=2$, d.o.f $K$, width $\sim K^2$, and 
  depth $L=5$, such
  that for the solution to \eqref{eq:ivp} with initial condition $u_0 \in \cU$ at
  each time $u(t) := u(\cdot, t)$ there is a $h (t) := h(\gamma(t), \theta(t)) \in H_r$
  with the approximation error
  \begin{equation}
    \Norm{ u(t) - h (t)}{L^1(\Domx)}
    \lesssim
    \frac{1}{K} \SemiNormlr{u_0}{TV} \left(1 + \SemiNormlr{u_0}{TV} \right)
    \left( 1 + T \Norm{F''}{L^\infty(u_0(\Domx))} \right).
  \end{equation}
  where the coefficients $\gamma(t), \theta(t)$ are linear or constant functions of $t$.
\end{theorem}

\begin{proof} See Appendix~\ref{proof:thm:lrnr_entropy}. \end{proof}

Note that neither the rank $r$ nor the number of layers $L$ depends on the
initial condition. This result shows there is an upper bound on a kind of
``nonlinear width'' as described by the number of coefficients $r$ and the
number of layers $L$, namely the tuple $(r, L) = (2, 5)$ associated with
the entropy solutions of the scalar conservation law \eqref{eq:ivp}
regardless of the complexity of the initial data or the shock topology
structure.

\section*{Acknowledgements}
The authors thank Randall J. LeVeque for helpful
discussions, and also thank Ari Stern for suggesting the term
\emph{shock time function}. Part of this work was done while D.R. was enjoying
the hospitality of the International Research Institute of Disaster Science
(IRIDeS) at Tohoku University through an invitation from Kenjiro Terada,
supported by Invitational Fellowship Program for Collaborative Research with
International Researcher (FY2023).

\appendix

\section{Glossary of definitions and relations} \label{sec:glossary}

\emph{Formulations of the entropy solution (Sec. 2).}
\begin{enumerate}[label=\thesection-\arabic*., topsep=8pt, leftmargin=30pt]
  \item $u$ entropy solution  (Sec.~\ref{sec:entropy_solution})
  \item $u_0$ initial condition to the entropy solution  (Sec.~\ref{sec:entropy_solution})
  \item $\xi$ mapping for similarity solution \eqref{eq:xi} 
  \item $X$ characteristic curves \eqref{eq:charcurves} 
  \begin{equation}
    u(x, t)
    =
    \left\{
    \begin{aligned}
      u_0 ( \charX^+ (x, t)) 
      \quad & (x, t) \in \Dom \smallsetminus \cR
      \\
      (F')^{-1} \left( \xi(x,t) \right)
      \quad & (x, t) \in \cR
    \end{aligned}
    \right.
    \tag{\ref{eq:entropy}}
  \end{equation}
  \item $\hat{u}_0$ an extension of the initial condition $u_0$ \eqref{eq:uhat}
  \item $\hcharX$  rarefied characteristic curves  \eqref{eq:rarefied_charcurves} 
  \begin{equation}
    u(x, t) = \hat{u}_0 \circ \hcharX^+(x,t) \quad (x,t) \in \Domx
    \tag{\ref{eq:u0Yp}}
  \end{equation}
  \item $\hcharX_0$ rarefied characteristics with multivalued inverse
  \begin{equation}
  \hcharX_0(x, t) := \hat{I}(x) + t (F'\circ \hat{u}_0) (x)
  \quad
    (\text{where }\hat{I} \text{ is given in \eqref{eq:hId}})
  \tag{\ref{eq:hcharX0}}
  \end{equation}
  \item $\rho = \mathbf{1}_{\RR_+}$ threshold function 
  \item $\lambda$ shock time function  (Sec.~\ref{sec:shock_time})
  \item $\chkcharX$ relief characteristic curves \eqref{eq:relief} 
  \begin{equation}
    u(x, t) = \hat{u}_0 \circ \chkcharX^+ (x, t) \quad (x,t) \in \Dom \tag{\ref{eq:ucts}}
  \end{equation}
\end{enumerate}

\noindent \emph{Low rank neural representation (LRNR) (Sec. 3).}

\begin{enumerate}[label=\thesection-\arabic*., topsep=8pt, leftmargin=30pt]
  \setcounter{enumi}{10}
  \item $h(x; W, B)$ feedforward neural network (Def.~\ref{def:nn})
  \item $\sigma$ ReLU activation (Def.~\ref{def:nn})
  \item $L$ depth of a neural network (Def.~\ref{def:nn})
  \item $M$ width of a neural network (Def.~\ref{def:nn})
  \item $H_r$ low rank neural representation  of rank $r$ (Def~\ref{def:LRNR})
  \item $K$ degree of freedom of a LRNR (Def.~\ref{def:LRNR})
  \item $\WW_r$, $\BB_r$ low dimensional weights and biases
  \begin{equation}
    \WW_r
    =
    \bigtimes_{\ell = 1}^L (C^\ell \cdot \UU^\ell_r) \subset \WW
    \qquad
    \BB_r
    =
    \bigtimes_{\ell = 1}^L (D^\ell \cdot \VV^\ell_r) \subset \BB
    \tag{\ref{eq:wrbr}}
  \end{equation}
  \item LRNR (Def.~\ref{def:LRNR})
  \begin{equation}
    H_r 
    = 
    \left\{
    h (W, B) \mid W \in \WW_r,\ B \in \BB_r 
    \right\}
    \tag{\ref{eq:lrnr}}
  \end{equation}
\end{enumerate}

\noindent \emph{Transported subspaces, inverse-bias trick, LRNR approximation of classical solutions (Sec. 4).}

\begin{enumerate}[label=\thesection-\arabic*., topsep=8pt, leftmargin=30pt]
  \setcounter{enumi}{18}
  \item $\bar{H}_r := (B \cdot \Psi_r) \circ (A \cdot \Phi_r)^{-1}$
  transported subspace (Def.~\ref{def:transported_subspace})
  \item $\rho_\eps$ continuous piecewise linear approximation to the
  threshold function \eqref{eq:rho-eps}
  \item Inverse-bias trick for $(f, g) \in B \cdot \Psi_r \times A \cdot \Phi_r$
  \begin{equation}
    f \circ g^{-1} (x)
    = 
    c_0 + \sum_{k} c_k \rho( g^{-1}(x) - x_k)
    = 
    c_0 + \sum_{k} c_k \rho( x - g(x_k))
    \tag{\ref{eq:inv-bias}}
  \end{equation}
  \item $f_\eps \circ ^+ g_\eps$ continuous (finitely) piecewise linear
  analogue of the inverse bias trick
  \begin{equation}
    f_\eps \circ^+ g_\eps (x)
    =
    \sum_{i=1}^N c_i \rho_\eps\left(x - g_\eps(x_i)\right)
    \tag{\ref{eq:cts_ibtrick}}
  \end{equation}
  \item LRNR ($L=2, r=3$) approximation $h(t) \in H_r$ to the classical solution (Thm.~\ref{thm:lrnr_classical})
  \begin{equation}
    \Norm{ u(t) - h(t) }{L^1(\Domx)}
    \lesssim
    \frac{1}{K}
    \abs{u_0}_{TV(\Domx)}
    \left(
    1 + T \Norm{F''}{L^\infty(u_0(\Domx))} \Norm{u_0'}{L^\infty(\Domx)}
    \right)
    \tag{\ref{eq:LRNR_classical}}
  \end{equation}
\end{enumerate}

  \noindent \emph{LRNR approximation the entropy solution (Sec. 5).}
\begin{enumerate}[label=\thesection-\arabic*., topsep=8pt, leftmargin=30pt]
  \setcounter{enumi}{23}
  \item $\skleft,\, \skright$ left and right shock endpoints \eqref{eq:sklr}
  \item $t_k$ uniform time-grid $t_k = (k-1) T / K$
  \item $\Xi_k$ time-layered approximation $\Xi_k (x, t)
    = \hId(x) + t \, v_k(x)$ \eqref{eq:Xik} for the time window $[t_k, t_{k+1}]$
  \begin{equation}
    \Norm{\hcharX(\cdot, t) - \Xi_k(\cdot, t)}{L^\infty(\hDomx)}
    \le
    \frac{T}{K}
    \SemiNormlr{u_0}{TV(\Domx)} \Norm{F''}{L^\infty(u_0(\Domx))}
    \tag{\ref{eq:hchar_Xi_err}}
  \end{equation}
  \item $\barwedge_k$ spatial discretization of the time-layered
  approximation $\Xi_k$ for the time window $[t_k, t_{k+1}]$ satisfying
  \begin{align*}
    &\Norm{\Xi_k(\cdot, t) - \barwedge_k(\cdot, t)}{L^\infty(\hDomx)}
    \le
    \frac{1}{K}
    \left( 1 + \Const \right)
    \tag{\ref{eq:Zk_error}}
    \\
    &\barwedge_k( z_i) = 
    \left\{
    \begin{aligned}
      \hcharX_0(z_i, t_k)
        & 
      \Ffor z_i \in \hDomx \setminus \shockset
      \\
      \hcharX ( z_i, t_k )
        & 
      \Ffor z_i \in \shockset
    \end{aligned}
    \right.
  \tag{\ref{eq:Zk_interp}}
  \end{align*}
  \item $\eta_k$ approximation of $v_{k+1} - v_{k}$ appearing inside
  $\Xi_k$ \eqref{eq:etak} satisfying
  \begin{equation}
  \Norm{t\eta_k}{L^\infty(\hDomx)} \le \frac{4 \Const}{K}
  \tag{\ref{eq:etakinfnorm}}
  \end{equation}
  \item $\doublebarwedge$ continuous piecewise linear
  approximation of the rarefied characteristics $\hcharX$
  \begin{equation}
    \Norm{\hcharX(\cdot, t) - \doublebarwedge(\cdot, t)}{L^\infty(\hDomx)}
    \lesssim
    \frac{1}{K}
    \left(
    1 + T \SemiNormlr{u_0}{TV(\Domx)} \Norm{F''}{L^\infty(u_0(\Domx))}
    \right)
    \tag{\ref{eq:hcharX_error}}
  \end{equation}
  \item $\hat{u}_{0\eps}$ continuous piecewise approximation of the extended
  initial condition $\hat{u_0}$ \eqref{eq:hatu0eps}
  \item $\bar{h}$ continuous piecewise linear approximation
    $\bar{h}(x, t)
    :=
    \hat{u}_{0\eps} \circ^+ \doublebarwedge (x, t)$ \eqref{eq:pwlin_entropy}
    satisfying
    \begin{equation}
      \begin{aligned}
      & \Norm{u(\cdot, t) - \bar{h}(\cdot,t))}{L^1(\Dom_x)}
      \\
      & \quad \lesssim
      \frac{1}{K} \SemiNormlr{u_0}{TV(\Domx)}
      \left( 1 + \SemiNormlr{u_0}{TV(\Domx)} \right)
      \left( 1 + T \Norm{F''}{L^\infty(u_0(\Domx))} \right)
      \end{aligned}
      \tag{\ref{eq:entropy_approx_error}} 
    \end{equation}
    \item LRNR ($L=5, r=2$) approximation $h(t) \in H_r$ to the entropy solution (Thm.~\ref{thm:lrnr_entropy})
    \begin{equation}
      \Norm{ u(t) - h (t)}{L^1(\Domx)}
      \lesssim
      \frac{1}{K} \SemiNormlr{u_0}{TV} \left(1 + \SemiNormlr{u_0}{TV} \right)
      \left( 1 + T \Norm{F''}{L^\infty(u_0(\Domx))} \right)
      \tag{\ref{eq:LRNR_classical}}
    \end{equation}
  
\end{enumerate}

\section{Proofs of theorems and lemmas} \label{sec:proofs}

  \subsection{Proof of Lemma~\ref{lem:rarefied}} \label{proof:lem:rarefied}
    The characteristics $\hcharX$ restricted to $\Domx$ by the expression
    $\hcharX(\iota(x), t)$ satisfy the same differential equation as $X$ so they
    necessarily agree, 
    \begin{equation} \label{eq:rarefied_entropy}
      \hcharX^+(x, t) = \iota \circ X^+(x, t)
      \quad
      \text{ for }
      (x, t) \in \Dom \smallsetminus \cR.
    \end{equation}
    Otherwise if $(x, t) \in \cR$, the expression yields the similarity
    solution: suppose $y \in \hDomx$ satisfies $x = \hat{I}(y) + t F'
      \circ \hat{u}_0(y)$
    then one can show that 
    $(x - \bar{\xi}(x,t))/ t = F' \circ \hat{u}_0(y)$,
    implying $\hat{u}_0 (y) = (F')^{-1} [(x- \bar{\xi}(x,t))/t]$. 
    This agrees with 
    the entropy solution \eqref{eq:entropy}.

  \subsection{Proof of Lemma~\ref{lem:compose_error}} \label{proof:lem:compose_error}
    Let us define,
    \begin{equation}
      f_0(x) := \sum_{i=1}^N c_i \rho(x - x_i),
      \qquad
      f_0 \circ^+ g_\eps(x) := \sum_{i=1}^N c_i \rho(x - g_\eps(x_i)).
    \end{equation}
    By the triangle inequality, we have
    \begin{equation}
      \begin{aligned}
        \Norm{f \circ g^+ - f_\eps \circ^+ g_\eps}{L^1(\Domx)}
          & \le
        \Norm{f \circ g^+ - f_0 \circ g^+}{L^1(\Domx)}
        \\
          & +
        \Norm{f_0 \circ g^+ - f_0 \circ^+ g_\eps}{L^1(\Domx)}
        \\
          & +
        \Norm{f_0 \circ^+ g_\eps - f_\eps \circ^+ g_\eps}{L^1(\Domx)}.
      \end{aligned}
    \end{equation}
    For the first term,
    \begin{equation}
      \begin{aligned}
        \Norm{f \circ g^+ - f_0 \circ g^+}{L^1(\Domx)}
          & =
        \int_{\Domx} \abs{f ( g^+(x)) - f_0 ( g^+(x))} \dx
        \\
          & =
        \int_{\hDomx} \abs{f ( y) - f_0 (y)} \abs{g'(y)} \dy
        \\
          & =
        \Norm{g'}{L^\infty(\hDomx)} \Norm{f - f_0}{L^1(\hDomx)},
      \end{aligned}
    \end{equation}
    then for the second term,
    \begin{equation}
      \begin{aligned}
          & \Norm{f_0 \circ g^+ - f_0 \circ^+ g_\eps}{L^1(\Domx)}
        \\
          & =
        \int_{\Domx}
        \abs{ \sum_{i=1}^N c_i \rho(x - g(x_i))
          -
          \sum_{i=1}^N c_i \rho(x - g_\eps(x_i))} \dx
        \\
          & \le
        \sum_{i=1}^N \abs{c_i}
        \int_{\Domx} \abs{\rho(x - g(x_i)) - \rho (x - g_\eps(x_i))} \dx
        \\
          & \le
        \left( \sum_{i=1}^N \abs{c_i} \right)
        \Norm{g - g_\eps}{L^\infty(\cX)}
        \le
        \SemiNormlr{f_0}{TV(\hDomx)} \Norm{g - g_\eps}{L^\infty(\cX)}.
      \end{aligned}
    \end{equation}
    Next, we have that
    \begin{equation}
      \begin{aligned}
          & \Norm{f_0 \circ^+ g_\eps - f_\eps \circ^+ g_\eps}{L^1(\Domx)}
        \\
          & =
        \int_{\Domx}
        \abs{\sum_{i=1}^N c_i \rho(x - g_\eps(x_i))
          - \sum_{i=1}^N c_i \rho_\eps(x - g_\eps(x_i))} \dx
        \\
          & \le \sum_{i=1}^N \abs{c_i}
        \int_{\Domx}
        \abs{\rho (x - g_\eps(x_i)) - \rho_\eps(x - g_\eps(x_i))} \dx
        \\
          & \le \half \left(\sum_{i=1}^N \abs{c_i} \right) \Norm{\rho - \rho_\eps}{L^1(\RR)}
        \le \SemiNormlr{f_0}{TV(\hDomx)} \Norm{\rho - \rho_\eps}{L^1(\RR)}.
      \end{aligned}
    \end{equation}
    Now $\SemiNormlr{f_0}{TV(\hDomx)} \le \SemiNormlr{f_\eps}{TV(\hDomx)} \le
      \SemiNormlr{f}{TV(\hDomx)}$ for $\eps$ sufficiently small, and further
    \begin{equation}
      \begin{aligned}
        \Norm{f - f_0}{L^1(\hDomx)}
          & \le
        \Norm{f - f_\eps}{L^1(\hDomx)}
        +
        \Norm{f_\eps - f_0}{L^1(\hDomx)}
        \\
          & \le
        \Norm{f - f_\eps}{L^1(\hDomx)}
        +
        \SemiNormlr{f}{TV(\hDomx)} \Norm{\rho - \rho_\eps}{L^1(\hDomx)}.
      \end{aligned}
    \end{equation}
    Putting together these inequalities, 
    \begin{equation}
      \begin{aligned}
          & \Norm{f \circ g^+ - f_\eps \circ^+ g_\eps}{L^1(\Domx)}
        \\
          & \le
        \SemiNormlr{f}{TV(\Domx)}
        \left(
        \Norm{g - g_\eps}{L^\infty(\cX)}
        +
        \Norm{\rho - \rho_\eps}{L^1(\RR)}
        \right)
        \\                    & \quad +
                              \Norm{g'}{L^\infty(\hDomx)}
        \left(
                              \Norm{f - f_\eps}{L^1(\hDomx)}
        +
                              \SemiNormlr{f}{TV(\hDomx)} \Norm{\rho - \rho_\eps}{L^1(\RR)}
        \right),
      \end{aligned}
    \end{equation}
    and the result follows upon rearranging the terms.

\subsection{Proof of Theorem~\ref{lem:shallow_approx}} \label{proof:lem:shallow_approx}

  We will first write each member $\bar{h} \in \bar{H}_{r-1}$ as a NN.
  Then we will write the collection of these NNs as a LRNR.
  Let $\bar{h} = f(\beta) \circ g(\alpha)^+$ and write $f(\beta)$ as
  \begin{equation}
    f(\beta) = \sum_{i} \beta_i \psi_i = \sum_{k=1}^K \bar{c}_k(\beta) \rho(\,\cdot - x_k),
  \end{equation}
  for some coefficients $\bar{c}_k(\beta)$ that depend linearly on $\beta$, and
  $\cX := \{x_k\}_{k=1}^K = \cup_i \{x_{ij}\}_{j=1}^{N_i}$. Recall that
  $g(\alpha) = \sum_{i=1}^r \alpha_i \phi_i.$
  Then, for corresponding continuous piecewise linear $f_\eps(\beta)$ and $g_\eps(\alpha)$, we
  have by Lem.~\ref{lem:compose_error},
  \begin{equation}
    \begin{aligned}
        & \Norm{f(\beta) \circ g(\alpha)^+ - f_\eps(\beta) \circ^+ g_\eps(\alpha) }{L^1(\Domx)}
      \\
        & \le 
      \SemiNormlr{f(\beta)}{TV(\hDomx)}
      \left(
      \Norm{\rho - \rho_\eps}{L^1(\RR)} + \Norm{g(\beta) - g_\eps(\beta)}{L^\infty(\cX)}
      \right).
    \end{aligned}
  \end{equation}
  By choosing $g_\eps$ to interpolate the point values $g_\eps(\beta)(x_k) =
    g(\beta)(x_k)$ for $x_k \in \cX$, the second term vanishes, yielding
  \begin{equation}\label{eq:fg_fege}
    \Norm{f(\beta) \circ g(\alpha)^+ - f_\eps(\beta) \circ^+ g_\eps(\alpha) }{L^1(\Domx)}
    \le 
    \SemiNormlr{f(\beta)}{TV(\hDomx)} \Norm{\rho - \rho_\eps}{L^1(\RR)}.
  \end{equation}
  Now, we rewrite $f_\eps(\beta) \circ^+ g_\eps(\alpha)$ as a NN,
  \begin{equation}
    \begin{aligned}
        & f_\eps(\beta) \circ ^+ g_\eps(\alpha)(x)
      \\
        & = \sum_{k=1}^K \bar{c}_k(\beta) \rho_\eps( x - g(\alpha)(x_k))
      \\
        & = \sum_{k=1}^K \bar{c}_k(\beta) \frac{1}{\eps}
      \left[
        \sigma\left(
        x + \frac{\eps}{2} - g(\alpha)(x_k)
        \right) 
        - \sigma\left(
        x - \frac{\eps}{2} - g(\alpha)(x_k)
        \right)
        \right],
    \end{aligned}
  \end{equation}
  so one can write
  \begin{equation}
    f_\eps(\beta) \circ ^+ g_\eps(\alpha)(x)
    =
    W^2 \sigma( W^1 x  + B^1) + B^2, \qquad x \in \Domx,
  \end{equation}
  with the weights and biases
  \begin{equation}
    \begin{aligned}
      W^1 & = 
      \left[ 1 , 1, \cdots, 1 \right]^T \in \RR^{2K \times 1},
      \\
      W^2 & = 
      \left[
        \bar{c}_1(\beta), \bar{c}_1(\beta), \bar{c}_2(\beta), \bar{c}_2 (\beta), \cdots
        , \bar{c}_K(\beta) , \bar{c}_K(\beta)
        \right] \in \RR^{1 \times 2K},
      \\
      B^1 & = 
      \left[
        \frac{\eps}{2} - g(\alpha)(x_1) , - \frac{\eps}{2} - g(\alpha)(x_1) ,
        \cdots,
        \frac{\eps}{2} - g(\alpha)(x_K) , - \frac{\eps}{2} - g(\alpha)(x_K)
        \right]^T
      \in \RR^{2K \times 1},
      \\
      B^2 & = 0 \in \RR^{1 \times 1}.
    \end{aligned}
  \end{equation}
  We reorganize the NNs as 2-layer LRNR $H_{r}$ in the form \eqref{eq:2layer_LRNR}
  in Exmpl.~\ref{expl:2layer_LRNR}. 
  Let $\{e_i\}_{i=1}^r$ denote the standard basis of $\RR^{r}$.
  We write the weights $W^1, W^2$ as
  \begin{equation} \label{eq:W1W2}
    W^1 = e_1 \cdot \UU_r^1,
    \quad
    W^2
    =
    [\beta_1, \cdots , \beta_{r-1}, 0]\cdot \UU_r^2,
  \end{equation}
  where $\UU_r^1 \subset \RR^{2K \times 1}$ and $\UU_r^2 \subset \RR^{1 \times 2K}$
  are given by 
  \begin{equation}
    \begin{aligned}
      \UU_r^1
        & =
      ( U_i^1)_{i=1}^r,
        & 
        & 
      \left[U_i^1\right]_j := 1,
      \text{ for all }i, j,
      \\
      \UU_r^2
        & =
      ( U_i^2)_{i=1}^r,
        & 
        & 
      \begin{cases}
        \left[U_i^2\right]_j
        :=
        \bar{c}_{\lfloor j/2 \rfloor} (e_i)
          & \text{ for } i = 1, ...\,, r-1
        \\
        \left[U_i^2\right]_j
        :=
        0
          & \text{ for } i = r.
      \end{cases}
    \end{aligned}
  \end{equation}
  Write the biases $B^1, B^2$ as
  \begin{equation} \label{eq:B1B2}
    B^1
    =
    [\alpha_1, \alpha_2, \cdots, \alpha_{r-1}, 1] \cdot \VV^1,
    \quad
    B^2
    =
    0 \cdot \VV^2,
  \end{equation}
  where $\VV_r^1 \subset \RR^{2K \times 1}$ and $\VV_r^2 \subset \RR^{1 \times
      1}$ are defined by 
  \begin{equation}
    \begin{aligned}
      \VV_r^1 & = (V_i^1)_{i=1}^r,
              & 
              & 
      \begin{cases}
        \left[ V_i^1 \right]_j:= \psi_i (x_{\lfloor j /2 \rfloor})
          & \text{ for } i = 1, ...\,, r-1,
        \\
        \left[ V_i^1 \right]_j:= 0
          & \text{ for } i = r,
      \end{cases}
      \\
      \VV_r^2 & = (V_i^2)_{i=1}^r,
              & 
              & 
      V_i^1 := 0,
      \quad
      \text{ for all }i.
    \end{aligned}
  \end{equation}
  Defining the map $\mu$ over $A \times B$, whose mapped values $\mu(\alpha,
    \beta) = (\gamma, \theta)$ have individual entries $\gamma =
    (\gamma^1, \gamma^2)$ and $\theta = (\theta^1, \theta^2)$ that are
  given by
  \begin{equation} \label{eq:gamma_alpha}
    \begin{aligned}
      \gamma^1 & = e_1,
                & 
      \gamma^2 & = [\beta_1, \cdots , \beta_{r-1}, 0]^T,
      \\
      \theta^1 & = [\alpha_1 , \cdots, \alpha_{r-1}, 1]^T,
                & 
      \theta^2 & = 0,
    \end{aligned}
  \end{equation}
  the
  LRNR $h(\gamma, \theta)$ is defined as in Exmpl.~\ref{expl:2layer_LRNR} with the
  coefficients \eqref{eq:gamma_alpha} and weights and biases
  (\ref{eq:W1W2}-\ref{eq:B1B2}). By construction, it satisfies
  \begin{equation}\label{eq:fgh}
    h(\gamma, \theta)
    =
    f_\eps(\beta) \circ^+ g_\eps(\alpha).
  \end{equation}
  Plugging \eqref{eq:fgh} into \eqref{eq:fg_fege},
  \begin{equation}
    \begin{aligned}
      \Norm{f(\beta) \circ g(\alpha)^+ - h(\gamma, \theta) }{L^1(\Domx)}
      \le 
      \SemiNormlr{f(\beta)}{TV(\hDomx)} \Norm{\rho - \rho_\eps}{L^1(\RR)}.
    \end{aligned}
  \end{equation}

  The desired inequality follows from
  \begin{equation}
    \SemiNormlr{f(\beta)}{TV(\hDomx)}
    \le
    \SemiNormlr{\sum_i \beta_i \psi_i}{TV(\hDomx)}
    \le
    \max_i \abs{\beta_i} \sum_i \SemiNormlr{\psi_i}{TV(\hDomx)}.
  \end{equation}

\subsection{Proof of Lemma~\ref{lem:Zk}} \label{proof:lem:Zk}
  By construction we have for any $t_k \in \Domt$
  \begin{equation}
    \hcharX (\cdot, t_k) = \Xi_k(\cdot, t_k),
  \end{equation}
  and for $t \in [t_k, t_{k+1}]$
  \begin{equation}
    \hcharX (z, t) = \Xi_k(z, t)
    \quad
    \text{ for all } z \in \hDomx \setminus \Upsilon_t.
  \end{equation}
  Let us fix $x_0 \in \Upsilon_{t_k}$
  and recall $\hcharX_0$ is the rarefied characteristics with multivalued 
  inverse \eqref{eq:hcharX0}. Then we have 
  \begin{equation} \label{eq:hXZk_error_pre}
    \begin{aligned}
      E_k (t; x_0)
        & :=
      \sup_{x \in [\skleft(x_0,t), \skright(x_0,t)]}
      \Norm{\hcharX (x, t) - \Xi_k(x, t)}{}
      \\
        & \le
      \int_{t_k}^t  
      \sup_{\substack{
          z \in [\skleft(x_0, t), \skleft(x_0, t_k) ]
      \\
          w \in [\skright(x_0, t_k), \skright(x_0, t) ]
        }}
      \abs{\ddf{\hcharXc}{t}(z, t)
        - \ddf{\hcharXc}{t}(w, t)} \dtau
      \\                   & \le
                          \left(
                          \sup
                          _{
                            \substack{
                              z \in [\skleft(x_0, t), \skleft(x_0, t_k) ]
      \\
          w \in [\skright(x_0, t_k), \skright(x_0, t) ]
        }
      }
      \abs{
        (F' \circ \hat{u}_0) (z)
        -
        (F' \circ \hat{u}_0) (w)
      }
      \right)
      \cdot
      \abs{t - t_k}
      \\                    & \le
                            \Normlr{F'' }{L^\infty ([\underline{u}_{\text{loc}},
                                \overline{u}_{\text{loc}}])}
                            (\overline{u}_{\text{loc}} - \underline{u}_{\text{loc}})
      \cdot
                            \abs{t - t_k},
    \end{aligned}
  \end{equation}
  where we let
  \begin{equation}\label{eq:uloc_minmax}
    \begin{aligned}
      \overline{u}_{\text{loc}}
        & :=
      \sup_{\substack{z \in [\skleft(x_0, t), \skleft(x_0, t_k) ] \\
          \qquad \cup [\skright(x_0, t_k), \skright(x_0, t) ]}} u_0(z),
      \quad
      \\
      \underline{u}_{\text{loc}}
        & :=
      \inf_{\substack{z \in [\skleft(x_0, t), \skleft(x_0, t_k) ] \\
          \qquad \cup [\skright(x_0, t_k), \skright(x_0, t) ]}} u_0(z).
    \end{aligned}
  \end{equation}
  Since $[ \overline{u}_{\text{loc}},  \underline{u}_{\text{loc}}] \subset
    u_0(\Domx)$, we obtain the first estimate \eqref{eq:hXZk_error}.

  Next, we prove the second upper bound. By construction $\Xi_k(\cdot,
    t)$ equals $\hcharX_0(\cdot, t)$ within $\shocksett{t_k}$ so we have
  \begin{equation} \label{eq:ZkTV_split}
    \begin{aligned}
      \abs{\Xi_k (\cdot, t)}_{TV(\hDomx)}
        & \le
      \abs{\Xi_k (\cdot, t)}_{TV(\shocksett{t_k})}
      +
      \abs{\Xi_k (\cdot, t)}_{TV(\hDomx \setminus \shocksett{t_k})}
      \\                   & \le
                          \abs{\Xi_k (\cdot, t)}_{TV(\shocksett{t_k})}
                          +
                          \SemiNorm{\hcharX_0 (\cdot, t)}{TV(\hDomx \setminus \shocksett{t_k})}.
    \end{aligned}
  \end{equation}
  To estimate the first term on the RHS, note that 
  $\Xi_k(\cdot, t)$ restricted
  to $\shocksett{t_k}$ is a linear interpolant of $\hcharX_0(\cdot, t)$. More
  specifically, for any $x_0 \in \shocksett{t_k}$,
  \begin{equation}
    \begin{aligned}
      \abs{\Xi_k(\cdot, t) }_{TV([\skleft(x_0, t_k), \skright(x_0, t_k)])}
        & =
      \abs{
        \hcharX_0(\skleft(x_0, t_k), t)
        -
        \hcharX_0(\skright(x_0, t_k), t)
      }
      \\
        & \le 
      \SemiNorm{\hcharX_0(\cdot, t) }
      {TV([\skleft(x_0, t_k), \skright(x_0, t_k)])},
    \end{aligned}
  \end{equation}
  and the local total variation bound is then naturally extended to $\shocksett{t_k}$,
  \begin{equation}
    \SemiNorm{\Xi_k(\cdot, t) }{TV(\shocksett{t_k})}
    \le
    \SemiNorm{\hcharX_0(\cdot, t) }{TV(\shocksett{t_k})}.
  \end{equation}
  Using this in \eqref{eq:ZkTV_split}, we have that
  \begin{equation}
    \begin{aligned}
      \abs{\Xi_k(\cdot, t) }_{TV(\hDomx)}
        & \le
      \SemiNorm{\hcharX_0(\cdot, t) }{TV(\hDomx)}
      =
      \SemiNorm{\hId + t F' \circ \hat{u}_0 }{TV(\hDomx)}
      \\
        & \le
      \abs{\Domx}
      +
      T
      \SemiNorm{u_0}{TV(\Domx)} \Norm{F''}{L^\infty(u_0(\Domx))},
    \end{aligned}
  \end{equation}
  which is the desired inequality upon inserting $\abs{\Domx}=1$.

    \subsubsection{Remark}
    
  Note that the upper bound \eqref{eq:hXZk_error} with the TV norm
  $\SemiNormlr{u_0}{TV(\Domx)}$ obtained from the last term in
  \eqref{eq:hXZk_error_pre} is not sharp, and can be replaced by $2
    \Norm{u_0}{L^\infty(\Domx)}$ because
    $\abs{ \overline{u}_{\text{loc}} - \underline{u}_{\text{loc}} }
    \le
    2\Norm{u_0}{L^\infty( [\skleft(x_0, t), \skright(x_0, t) ])},$
  but we will keep the TV norm here, as doing so this will not affect the final
  results of this work.

\subsection{Proof of Theorem~\ref{thm:lrnr_classical}} \label{proof:thm:lrnr_classical}
  We write the classical solution $u(t) = u(\cdot, t)$ as
  \begin{equation}
    \begin{aligned}
      u(\cdot, t) & = u_0 \circ X^{-1}(\cdot,t) 
      \\                    & =
                            u_0 \circ (\Id(\cdot) + t (F' \circ u_0) (\cdot))^{-1},
    \end{aligned}
  \end{equation}
  then it is clear that the set of solutions forms a transported subspace,
  so Thm.~\ref{lem:shallow_approx} applies. Recall that the approximation proceeds
  below:
  Letting $\hDomx = \Domx$, define $u_{0\eps}, X_\eps \in C(\Domx) \cap
    P_1(\Domx)$ for which Lem.~\ref{lem:compose_error} implies
  \begin{equation}
    \begin{aligned}
        & \Norm{u_0 \circ X^+ - u_{0\eps} \circ^+ X_\eps}{L^1(\Domx)}
      \\
        & \le
      \Norm{X'}{L^\infty (\Domx)} \Norm{u_0 - u_{0\eps}}{L^1(\Domx)}
      \\
        & +
      \SemiNormlr{u_0}{TV(\Domx)} \left(
      (1 + \Norm{X'}{L^\infty(\Domx)}) \Norm{\rho - \rho_\eps}{L^1(\RR)}
      + \Norm{X - X_\eps}{L^\infty(\cX)}
      \right),
    \end{aligned}
  \end{equation}
  then picking $X_\eps$ so that it interpolates $X$
  at the points in $\cX
    \subset \Domx$ as set in Lem.~\ref{lem:compose_error}, the norm $\Norm{X -
      X_\eps}{L^\infty(\cX)} = 0$, so we have
  \begin{equation}\label{eq:classic_subspace_error}
    \begin{aligned}
        & \Norm{u_0 \circ X^+ - u_{0\eps} \circ^+ X_\eps}{L^1(\Domx)}
      \\
        & \le
      \Norm{X'}{L^\infty (\Domx)} \Norm{u_0 - u_{0\eps}}{L^1(\Domx)}
      +
      \SemiNormlr{u_0}{TV(\Domx)}
      (1 + \Norm{X'}{L^\infty(\Domx)}) \Norm{\rho - \rho_\eps}{L^1(\RR)}.
    \end{aligned}
  \end{equation}
  As shown in \eqref{eq:fgh} in the proof of Thm.~\ref{lem:shallow_approx},
  there is a LRNR $h$ satisfying 
  \begin{equation}\label{eq:classicsubspace_LRNR_error}
    \begin{aligned}
        & \Norm{u_{0\eps} \circ^+ X_\eps(t) - h (t)}{L^1(\Domx)}
      \\
        & \le
      T  \abs{u_0}_{TV(\Domx)} \Norm{\rho - \rho_\eps}{L^1(\RR)}
      +
      \frac{1}{K} \Norm{X'}{L^\infty (\Domx)} \abs{u_0}_{TV(\Domx)},
    \end{aligned}
  \end{equation}
  due to the estimate 
  \begin{equation}\label{eq:hcharXprime_classical}
    \Norm{X'}{L^\infty(\Domx)}
    \le
    1 + T\Norm{F''}{L^\infty(u_0(\Domx))} \Norm{u'}{L^\infty(\Domx)}.
  \end{equation}
  Setting $\eps$ to be sufficiently small, the result follows by \eqref{eq:classic_subspace_error} and \eqref{eq:classicsubspace_LRNR_error}.

  The coefficients $\gamma, \theta$ are induced by the mapping $\mu$
  given by Thm.~\ref{lem:shallow_approx}, see \eqref{eq:gamma_alpha}. The
  transported subspace coefficients depend linearly on $t$ or are
    constant respect to it, thus $\gamma, \theta$ are linear or
    constant functions of $t$.

\subsection{Proof of Theorem~\ref{lem:chieut}} \label{proof:lem:chieut}
  By construction, $\doublebarwedge(\cdot, t)$ is based on $\eta_k
    \in C(\hDomx) \cap P_1(\hDomx)$ given in \eqref{eq:etak} for index $k = 1,
    ...\,,K-1$ in \eqref{eq:chieut_approx} so we have for $t \in [t_k, t_{k+1}]$
  \begin{equation}
    \begin{aligned}
        & \Normlr{\hcharX(\cdot, t) - \doublebarwedge(\cdot, t)}{L^\infty(\hDomx)}
      \\
        & \le
      \Normlr{\barwedge(\cdot, t) - \doublebarwedge(\cdot, t)}{L^\infty(\hDomx)}
      +
      \Normlr{\hcharX(\cdot, t) - \barwedge(\cdot, t)}{L^\infty(\hDomx)}
      \\
        & \le
      \Normlr{t \eta_k}{L^\infty(\hDomx)}
      +
      \Normlr{\hcharX(\cdot, t) - \barwedge(\cdot, t)}{L^\infty(\hDomx)}
      \\
        & \le
      \frac{4 \Const}{K}
      +
      \frac{1 + 6\Const}{K}
      \le
      \frac{1 + 10 \Const}{K},
    \end{aligned}
  \end{equation}
  upon applying \eqref{eq:etakinfnorm} and \eqref{eq:barwedge_error}.
  The RHS does not depend on the interval $[t_k, t_{k+1}]$ that $t$ belongs to,
  so the result follows.

\subsection{Proof of Theorem~\ref{thm:pwlin_entropy}} \label{proof:thm:pwlin_entropy}
  We initially suppose $u_0 \in \overline{\cU}$ so that $u_0 = \bar{u}_0$.
  We begin by providing an upper bound on the derivative of
  $\hcharX$.  Recall that $\hcharX$ is continuous and piecewise smooth, therefore its
  weak derivative $\hcharX'$ is in $L^\infty(\hDomx)$. Decomposing the
  extended domain into the shock set $\shockset$ and its complement 
  \begin{equation}
    \hDomx = (\hDomx \setminus \shockset) \cup \shockset,
  \end{equation}
  we have that $\hcharX$ is constant in $\shockset$ so that 
  $\Norm{\hcharX'}{L^\infty(\shockset)} = 0$. 
  Now recall the intervals inserted into in the extended domain 
  $\Gamma_x = \bigcup_{k=1}^{n_\cR}\Gamma_k$ defined in \eqref{eq:Gammax}.
  Writing 
  the shock-free part of the extended domain $\hDomx \setminus \shockset$
  into $\Gamma_x$ and its complement
  \begin{equation}
    \hDomx \setminus \shockset
    =
    \left((\hDomx \setminus \shockset) \, \setminus \Gamma_x\right) \cup \Gamma_x,
  \end{equation}
  we straightforwardly bound,
  \begin{equation}
    \begin{aligned}
      \Norm{\hcharX'}{L^\infty(\hDomx)}
        & \le
      \Norm{\hcharX'}{L^\infty(\hDomx \setminus \shockset)}
      +
      \Norm{\hcharX'}{L^\infty(\shockset)}
      \\                    & \le
                            \Norm{\hcharX'}{L^\infty(((\hDomx \setminus \shockset) \setminus \Gamma_x)
                              \cup (\Gamma_x \setminus \shockset))}
      \\                    & \le
                            1 + T \max\left\{
                            \Norm{(F' \circ \hat{u}_0)'}
                            {L^\infty((\hDomx \setminus \shockset) \setminus \Gamma_x)}
      ,\,
                            \Norm{(F' \circ \hat{u}_0)'}{L^\infty(\Gamma_x)}
      \right\}.
    \end{aligned}
  \end{equation}
  Due to
  the way of $\hat{u}_0$ was defined in \eqref{eq:uhat},
  for $k = 1, ...\,,n_\cR$,
  \begin{equation}
    \begin{aligned}
      \Norm{(F' \circ \hat{u}_0)'}{L^\infty(\Gamma_k)}
        & =
      \Normlr{\nu_k^- + \frac{\nu_k^+ - \nu_k^-}{\eps_k} (\cdot - \iota(\xi_k))}
      {L^\infty(\iota(\xi_k) + [0, \eps_k))}
      \\                    & = \max_{\pm}\abs{\nu_k^\pm} \le \Norm{(F' \circ u_0)'}{L^\infty(\Domx)},
    \end{aligned}
  \end{equation}
  since the $L^\infty$-norm is taken over a linear function
  connecting $\nu_k^-$ and $\nu_k^+$.
  So we obtain the $L^\infty$ estimate
  \begin{equation} \label{eq:hcharXprime}
    \Norm{\hcharX'}{L^\infty(\hDomx)}
    \le
    1 + T \Norm{(F' \circ u_0)'} {L^\infty(\Domx)}
    \le
    1 + T \Norm{F''}{L^\infty(u_0(\Domx))}
    \SemiNormlr{u_0}{TV(\Domx)},
  \end{equation}
  which is a generalization of the crude bound \eqref{eq:hcharXprime_classical}
  for classical characteristics.

  Due to Lem.~\ref{lem:compose_error}, and choosing $\epsilon = 1/K$ so that
  $\Norm{\rho - \rho_\eps}{L^1(\RR)} \lesssim 1/K$ we have
  \begin{equation}
    \begin{aligned}
        & \Norm{u(\cdot, t) - \bar{h} (\cdot, t)}{L^1(\Domx)}
      \\
        & \quad \le
      \Norm{\hcharX'}{L^\infty(\hDomx)} \Norm{\hat{u}_0 - \hat{u}_{0\eps}}{L^1(\hDomx)}
      \\
        & \qquad+
      \SemiNormlr{u_0}{TV(\Domx)}
      \left(
      \left( 1 + \Norm{\hcharX'}{L^\infty(\hDomx)} \right)
      \Norm{\rho - \rho_\eps}{L^1(\RR)}
      +
      \Norm{\hcharX - \doublebarwedge}{L^1(\RR)}
      \right).
    \end{aligned}
  \end{equation}
  The result follows from the estimates \eqref{eq:hcharX_error} and
  \eqref{eq:hcharXprime}, and further,
  \begin{equation}
      \Norm{u(\cdot, t) - \bar{h} (\cdot,t))}{L^1(\Dom_x)}
      \lesssim
      \frac{1}{K} \SemiNormlr{u_0}{TV}
      +
      \frac{1}{K} \SemiNormlr{u_0}{TV}^2 T \Norm{F''}{L^\infty(u_0(\Dom_x))}.
  \end{equation}
  Gathering the terms,
  \begin{equation}
    \begin{aligned}
        & \Norm{u(\cdot, t) - \bar{h} (\cdot,t))}{L^1(\Dom_x)}
      \\                   & \lesssim \frac{1}{K} \SemiNormlr{u_0}{TV}
      \left(1 + \SemiNormlr{u_0}{TV} \right)
      \left(1 + T \Norm{F''}{L^\infty(u_0(\Domx))} \right).
    \end{aligned}
  \end{equation}
  Finally, lifting the restriction that $u_0 \in \overline{\cU}$,
  then employing a density argument followed by an application of the triangle
  inequality, the desired estimate is obtained.

\subsection{Proof of Theorem~\ref{thm:lrnr_entropy}} \label{proof:thm:lrnr_entropy}
    Each solution to the IVP with $u_0 \in \cU$ can be 
    uniformly approximated by
    one with $\bar{u}_0 \in  \overline{\cU}$, so it suffices to consider 
    that case. Recall the approximation of the initial 
    data of the form $\hat{u}_{0\eps}$ \eqref{eq:hatu0eps}
    with its grid $(x_i)_{i=1}^K$ and coefficients $(\hat{u}_{0,
        i})_{i=1}^K$, along with the approximation $\doublebarwedge$
    \eqref{eq:chieut_approx} of the rarefied characteristics.

    We will show that the piecewise approximation $h(x,t)$
    \eqref{eq:pwlin_entropy} can be represented by a LRNR. Denote the
    hidden variables by
    \begin{equation}
      Z^{\ell+1} = W^{\ell} \sigma(Z^{\ell}) + B^{\ell},
      \quad
      \ell = 0, ... , L-1,
    \end{equation}
    and the initial input is $Z^{0} := x$. We remind the reader that in our
    setting, the spatial variable $x$ is the sole input to the NN, and the time variable $t$ is an external parameter. We use the notation $[N] := \{1, 2, ...\,, N\}$ for $N \in
      \NN$ and the Einstein summation convention for the subscripts. 
    We will use higher-dimensional tensors to specify the hidden variables
    and the affine mappings instead of using vectors and matrices, in order to
    simplify the index notation.

    We construct the first layer $(\ell = 1)$ so that we have
    \begin{equation}
      \RR^{K \times 3 \times K}
      \ni
      Z^{1}_{ijk}
      =
      \minialign{
        x,        && & i \in [K], j = k = 1,    \\
        x_i,      && & i \in [K], j = 2, k = 1, \\
        t - t_k,  && & i = 1, j = 3, k \in [K], \\
        0,         && & \text{otherwise.}
      }
    \end{equation}
    Note also that only $\sim K$ entries are non-zero. 
    We obtain this by 
    \begin{equation}
      W_{ijk}^{1}
      =
      \minialign{
        1, && &j, k = 1, i \in [K], \\
        0,         && & \text{otherwise,}
      }
    \end{equation}
    and
    \begin{equation}
      B_{ijk}^{1}
      =
      \minialign{
        x_i,      && &i \in [K], j = 2, k = 1,  \\
        t - t_k,  && &i = 1, j = 3, k \in [K], \\
        0,         && & \text{otherwise,}
      }
    \end{equation}
    $B$ here has linear or constant dependence on $t$ of rank two.

    We construct the second layer $(\ell = 2)$ so that we have
    \begin{equation}
      \RR^{K \times K \times 2}
      \ni
      Z^{2}_{pqn}
      =
      \minialign{
        x,   && & n = q = 1, p \in [K], \\
        x_p - z_q - \sigma \left(\frac{1}{\eps_0} (t - t_q) \right),
        && & n = 2, p, q \in [K],  \\
        0,         && & \text{otherwise,}
      }
    \end{equation}
    with $\eps_0 := T/K$, and recall that $(z_k)_{k=1}^K$ is the adaptive grid used for constructing
    $\Xi_k$ \eqref{eq:Xik}.
    We obtain this by assigning the weights 
    \begin{equation}
      W_{ijk,pqn}^{2}
      =
      \minialign{
        \delta_{ip},              && & s = 1, j= 1,\, i, p \in [K], n = 1,\\
        \delta_{ip},          && & s = 2, j = 2, \, i \in [K], n \in [K],\\
        -\frac{\delta_{kq}}{\eps_0},&& & s = 2, i = 1, j = 3,\, k,q \in [K],\\
        0,         && & \text{otherwise,}
      }
    \end{equation}
    and the bias 
    \begin{equation}
      B_{pqn}^{2}
      =
      \minialign{
        - z_q, && & n = 2, \, p \in [K],\, q \in [K], \\
        0,         && & \text{otherwise,}
      }
    \end{equation}
    and recalling that $\sigma(z) / \eps_0 = \sigma(z/ \eps_0)$. This layer
    has rank one, and the operations are identical for each $p$ in this and the
    subsequent layer, so it is parallelizeable.

    We construct the third layer $(\ell = 3)$ so that we have
    \begin{equation}
      \RR^{K \times 2}
      \ni
      Z^{3}_{ij}
      =
      \minialign{
        x, && & i \in [K], j = 1, \\
        x_i + \doublebarwedge(x_i, t), && & i \in [K], j = 2, \\
        0,         && & \text{otherwise,}
      }
    \end{equation}
    We obtain this by setting
    \begin{equation}
      W_{ij,pqn}^{3}
      =
      \minialign{
        \delta_{ip}, && & n = q = j = 1, i, p \in [K], \\
        a_{q}(t) \delta_{ip},
        && & j = n = 2,\, i,p \in [K], \\
        0,         && & \text{otherwise,}
      }
    \end{equation}
    and
    \begin{equation}
      B_{ij}^{3}
      =
      \minialign{
        0, && & j = 1,\\
        \barwedge_K(x_i, t), && & i \in [K], j = 2, \\
        0,         && & \text{otherwise,}
      }
    \end{equation}
    in which we set the coefficients $a_{q}(t)$, an linear or constant function of $t$,
    to obtain the the function \eqref{eq:chieut_approx} of above.
    This layer has rank two.

    We construct the fourth layer $(\ell = 4)$ so that we have 
    \begin{equation} \label{eq:eps1wgts}
      \RR^{K \times 2}
      \ni
      Z^{4}_{ps}
      =
      \minialign{
        \sigma\left( 
        -\frac{1}{\eps_1} (x - x_p - \doublebarwedge(x_p, t)) - \half
        \right),
        && & p \in [K], s = 1,
        \\
        \sigma \left( 
        \frac{1}{\eps_1} (x - x_p - \doublebarwedge(x_p, t)) + \half
        \right),
        && & p \in [K], s = 2, \\
        0,         && & \text{otherwise,}
      }
    \end{equation}
    where we set $\eps_1 < \min_{p, q \in [K], p \ne q}  |x_p - x_q|$. We
    obtain this by setting
    \begin{equation}
      W_{ps,ij}^{4}
      =
      \minialign{
        \frac{(-1)^{(j+s+1)}}{\eps_1} \delta_{pi} ,
        && &s = 1,2, \, i, p \in [K], \, j=1,2, \\
        0,         && & \text{otherwise,}
      }
    \end{equation}
    and
    \begin{equation}
      B_{ps}^{4}
      =
      (-1)^{(s+1)}/2,
      \quad s = 1,2,  p \in [K].
    \end{equation}
    The layer has rank one.

    We construct the fifth layer $(\ell = 5)$ so that we have  
    \begin{equation}
      h (x, t)
      =  \sum_{j = 1}^K \hat{u}_{0, j}
      \rho_\eps (x - \doublebarwedge(x_j, t)).
    \end{equation}
    We assign the weights that approximate
    the extended initial condition $\hat{u}_0$,
    \begin{equation}
      W_{ps}^{5}
      =
      (-1)^{s+1} \hat{u}_{0, p},
      \qquad
      s = 1,2, 
      \quad
      p \in [K],
    \end{equation}
    and then let $B^{5}$ be zero. These weights take linear combinations
    of the inputs from the previous layer $Z^4$ \eqref{eq:eps1wgts} so that 
    $\rho_\eps(x - \doublebarwedge(x_j, t))$ is computed; note how $\rho_\eps$
    was written as a sum of two ReLUs \eqref{eq:rho-eps}. The layer has rank one.

    Then we have that $h(x,t)$ is equal to $\bar{h}(x,t)$, so the error
    estimate follows from that for $\bar{h}$ (Thm.~\ref{thm:pwlin_entropy}).
    All the weights and biases can be written as a linear combination of
    at most two basis elements, and the coefficients are either constant
    or linear with respect to $t$.

\bibliographystyle{siamplain}

\end{document}